%% file: paper.tex
\documentclass[11pt]{article}
\usepackage[utf8]{inputenc}
\usepackage[T1]{fontenc}
\usepackage{graphicx,subcaption}
\usepackage{longtable,booktabs,tabularx}
\usepackage{wrapfig,rotating,capt-of,pdflscape}
\usepackage[normalem]{ulem}
\usepackage{amsmath,amsthm,amssymb,mathtools,mathrsfs,bm}
\usepackage[margin=1in]{geometry}
\usepackage{enumitem}
\usepackage{graph-separation}
\usepackage{tikz}
\usetikzlibrary{cd,decorations.pathmorphing,arrows.meta}
\usepackage{hyperref}
\usepackage{cleveref}
\hypersetup{pdftitle={On statistical and causal models associated with acyclic directed mixed graphs},pdfauthor={Qingyuan Zhao}}
\usepackage[
  backend=biber,
  style=authoryear-comp,
  citestyle=authoryear-comp,
  url=false,
  uniquename=false,
  sorting=ynt,
  sortcites
]{biblatex}
\setlist[enumerate]{noitemsep}
\setlist[description]{noitemsep}
\setlist[itemize]{noitemsep}
\theoremstyle{plain}
\newtheorem{theorem}{Theorem}
\newtheorem{proposition}{Proposition}
\newtheorem{lemma}{Lemma}

\theoremstyle{definition}
\newtheorem{definition}{Definition}
\theoremstyle{remark}

\newtheorem{example}{Example}
\DeclareMathAlphabet{\mathpzc}{OT1}{pzc}{m}{it}
\newcommand{\gG}{G}
\newcommand{\ingraph}[1]{~\textbf{in}~#1}

\newcommand{\textnot}{\textnormal{\bf not}~}
\newcommand{\underdist}[1]{~\textnormal{\bf under}\,#1}
\newcommand{\DMG}{\mathbb{G}}
\newcommand{\DMGc}{\mathbb{G}^{\ast}}
\newcommand{\DiG}{\mathbb{G}^{\ast}_{\textnormal{D}}}
\newcommand{\ADMG}{\mathbb{G}_{\textnormal{A}}}
\newcommand{\ADMGc}{\mathbb{G}^{\ast}_{\textnormal{A}}}
\newcommand{\DAG}{\mathbb{G}^{\ast}_{\textnormal{DA}}}
\newcommand{\UG}{\mathbb{UG}}
\newcommand{\BDG}{\mathbb{G}^{\ast}_{\textnormal{B}}}

\newcommand{\UnDMGc}{\mathbb{G}_{\textnormal{E}}^{\ast}}
\newcommand{\UnADMGc}{\mathbb{G}_{\textnormal{EA}}^{\ast}}
\newcommand{\statmodel}{\mathbb{P}}
\newcommand{\statmodelGM}{\mathbb{P}_{\textnormal{M}}}

\newcommand{\statmodelMI}{\mathbb{P}_{\textnormal{MI}}}
\newcommand{\statmodelE}{\mathbb{P}_{\textnormal{E}}}
\newcommand{\statmodelNE}{\mathbb{P}_{\textnormal{NE}}}
\newcommand{\statmodelPE}{\mathbb{P}_{\textnormal{PE}}}
\newcommand{\statmodelCE}{\mathbb{P}_{\textnormal{CE}}}
\newcommand{\statmodelLM}{\mathbb{P}_{\textnormal{LM}}}
\newcommand{\statmodelNM}{\mathbb{P}_{\textnormal{NM}}}
\newcommand{\statmodelF}{\mathbb{P}_{\textnormal{F}}}
\newcommand{\statmodelEF}{\mathbb{P}_{\textnormal{EF}}}
\newcommand{\statmodelA}{\mathbb{P}_{\textnormal{A}}}
\newcommand{\causalmodel}{\mathbb{CP}}

\let\P\relax
\DeclareMathOperator{\P}{\mathsf{P}}
\let\Q\relax
\DeclareMathOperator{\Q}{\mathsf{Q}}
\let\p\relax
\DeclareMathOperator{\p}{\mathsf{p}}

\DeclareMathOperator{\marg}{margin}
\DeclareMathOperator{\fix}{fix}

\DeclareMathOperator{\augg}{augment}
\DeclareMathOperator{\expand}{expand}
\DeclareMathOperator{\expandC}{expand_C}
\DeclareMathOperator{\expandP}{expand_P}
\DeclareMathOperator{\expandN}{expand_N}

\DeclareMathOperator{\sE}{\mathcal{E}}
\DeclareMathOperator{\sD}{\mathcal{D}}
\DeclareMathOperator{\sB}{\mathcal{B}}
\DeclareMathOperator{\sI}{\mathcal{I}}
\DeclareMathOperator{\sJ}{\mathcal{J}}

\DeclareMathOperator{\sK}{\mathcal{K}}

\DeclareMathOperator{\sL}{\mathcal{L}}
\DeclareMathOperator{\sM}{\mathcal{M}}

\DeclareMathOperator{\an}{\mathrm{an}}
\DeclareMathOperator{\de}{\mathrm{de}}
\DeclareMathOperator{\nd}{\mathrm{nd}}
\DeclareMathOperator{\pa}{\mathrm{pa}}
\DeclareMathOperator{\ch}{\mathrm{ch}}
\DeclareMathOperator{\mb}{\mathrm{mb}}
\DeclareMathOperator{\mbg}{\mathrm{mbg}}
\DeclareMathOperator{\dis}{\mathrm{dis}}

\newcommand\independent{\protect\mathpalette{\protect\independenT}{\perp}}
\def\independenT#1#2{\mathrel{\rlap{$#1#2$}\mkern2mu{#1#2}}}
\makeatletter
\providecommand{\leftsquigarrow}{\mathrel{\mathpalette\reflect@squig\relax}}
\newcommand{\reflect@squig}[2]{\reflectbox{$\m@th#1\rightsquigarrow$}}
\makeatother

\newcommand{\fitdisplay}[1]{\begingroup\setbox0=\hbox{$\displaystyle #1$}%
  \ifdim\wd0>\linewidth\resizebox{\linewidth}{!}{\box0}\else\box0\fi\endgroup}

\author{Qingyuan Zhao\thanks{Statistical Laboratory, University of Cambridge,
    qyzhao@statslab.cam.ac.uk.}}
\date{November 15, 2026}
\title{On statistical and causal models associated with acyclic
  directed mixed graphs}

\begin{document}
\maketitle
\begin{abstract}
  Causal models in statistics are often described using acyclic directed
  mixed graphs (ADMGs), which contain directed and bidirected edges
  and no directed cycles. This article surveys various interpretations
  of ADMGs, discusses their relations in different sub-classes of
  ADMGs, and argues that one of them---the noise expansion (NE)
  model---should be used as the default interpretation. Our endorsement of the NE model is based on two
  observations. First, in a subclass of ADMGs called
  exogenously bidirected graphs (which retain most of the good properties of
  directed acyclic graphs and bidirected graphs), the NE model is
  equivalent to many other interpretations, including the ordinary
  Markov and nested Markov models. Second, the
  NE model for an arbitrary ADMG is exactly the union of the observed
  margins of NE models on its exogenously bidirected expansions.
  This property is referred to as \emph{completeness}, as it
  shows that the model does not commit to any specific latent variable
  explanation. In proving the nested Markov property, we
  also develop an ADMG-based theory for causality and its complete characterization via the potential outcome calculus. Finally, we compare
  the NE model with the closely related but different interpretation
  of ADMGs as directed acyclic graphs (DAGs) with latent variables
  that is commonly used in the literature. We argue that the ``latent DAG'' interpretation is mathematically unnecessary, makes obscure
  ontological assumptions, and discourages practitioners from
  deliberating over important structural assumptions.
\end{abstract}


\input{main}

\section*{Acknowledgments}
The author thanks Wenjie Hu for discussion on the causal Markov model
for ADMGs, Thomas Richardson for constructive feedback, and Robin
Evans for pointing out a mistake in an earlier draft.

\section*{Funding}
This work is supported by the Engineering and Physical Sciences
Research Council (grant number EP/V049968/1).

\section*{Statement on AI use}
Generative AI was not used in previous versions but was used in preparing
this version to assist writing, check proofs and strengthen mathematical arguments.

\begin{refcontext}[sorting=nyt]
\printbibliography
\end{refcontext}

\clearpage
\appendix
\setcounter{equation}{0}
\renewcommand{\theequation}{A.\arabic{equation}}
\renewcommand{\theHequation}{appendix.\arabic{equation}}
\setcounter{lemma}{0}
\renewcommand{\thelemma}{A.\arabic{lemma}}
\renewcommand{\theHlemma}{appendix.\arabic{lemma}}
\setcounter{proposition}{0}
\renewcommand{\theproposition}{A.\arabic{proposition}}
\renewcommand{\theHproposition}{appendix.\arabic{proposition}}
\setcounter{example}{0}
\renewcommand{\theexample}{A.\arabic{example}}
\renewcommand{\theHexample}{appendix.\arabic{example}}
\setcounter{figure}{0}
\renewcommand{\thefigure}{A.\arabic{figure}}
\renewcommand{\theHfigure}{appendix.\arabic{figure}}
\input{appendix}
\end{document}

%% file: main.tex
\section{Introduction}
\label{sec:introduction}

Acyclic directed mixed graphs (ADMGs) were first used by
\textcite{wrightRelativeImportanceHeredity1920,wright34_method_path_coeff}
to describe causal relationships between a collection of random
variables in a linear model. They play a central role in
the modern statistical theory for causality; see, for example,
\textcite{pearlCausality2009} %
and
\textcite{richardsonNestedMarkovProperties2023}, although Pearl uses a
different terminology.
ADMGs have two types of edges---directed and bidirected. When
interpreting model assumptions encoded by ADMGs, two heuristics are
commonly used:
\begin{enumerate}
\item A directed edge means direct causal influence and a bidirected
  edge means exogenous correlation
  (\textcite{wright34_method_path_coeff} calls this ``residual
  correlation'').
\item An ADMG describes a latent variable model in the sense that the graphical
  structures
  \[\fitdisplay{
    V_1 \ldedge V_2 \rdedge V_3,~V_1 \bdedge V_2 \rdedge
    V_3,~\text{and}~V_1 \ldedge V_2 \bdedge V_3
  }\]
  all ``marginalize'' to $V_1 \bdedge V_3$ when we treat $V_2$ as
  unobserved.
\end{enumerate}
Graph marginalization, also called \emph{latent projection}
\parencite{vermaEquivalenceSynthesisCausal1990}, intends to describe the
implied model on the marginal distributions and will be formally
defined in \Cref{sec:marginalization}.

Based on these heuristics, many interpretations of ADMGs have been
proposed in the literature \parencite[see
e.g.][]{richardsonMarkovPropertiesAcyclic2003,petersElementsCausalInference2017,bareinboimPearlHierarchyFoundations2022}.
Unfortunately, these interpretations generally do not agree with each
other, and it is notoriously
difficult to describe the complicated constraints imposed by the
latent variables on the probability distribution of the observed
variables. The purpose of this article is to
give a survey of those interpretations of ADMGs, discuss their
relations, and put forward a case that one of those
interpretations---the nonparametric equation system (the E model
below)---should be used as the default. %
A key
argument is that the E model is \emph{complete} with respect to certain
latent variable explanations, a concept that will be defined shortly.
Another advantage is that causal reasoning can be based directly on
potential outcomes of the graph's own variables, without requiring a
latent DAG representation; see \Cref{thm:equivalent-axioms}.

Before moving to any technical discussion, it is useful to first
consider in what sense we can claim an interpretation is more natural
than others.
Generally speaking, an interesting mathematical definition
can be found in at least two ways:
\begin{enumerate}
\item Equivalence: when many definitions motivated by apparently different
  considerations are equivalent to each other, we may believe they
  describe a natural mathematical concept.
\item Completion: when there exists a natural definition for a smaller class of
  mathematical objects, we may try to find a ``completion'' of that
  definition to a larger class of objects.
\end{enumerate}
In fact, the Equivalence argument is regularly used in the graphical models
literature. A prominent example is the Hammersley-Clifford theorem,
which shows that two statistical models (of distributions with
positive densities) associated with an undirected
graph---one defined via factorization and another via the Markov
property---are equivalent \parencite[Theorem
3.9]{lauritzenGraphicalModels1996}. Another familiar example is the
equivalence of
the factorization model (``Bayesian networks'') and the Markov
models associated with directed acyclic graphs (DAGs)
\parencite[Theorem 3.27]{lauritzenGraphicalModels1996}. However, the
Equivalence argument by itself cannot define the ``right''
interpretation of ADMGs. In fact, we will see shortly that most common
interpretations of ADMGs in statistics are genuinely different. Given
this, one may be tempted to use the Completion argument instead. We will
see below that this is indeed possible but requires a careful
definition of completeness. %

\subsection{Directed mixed graphs}
\label{sec:direct-mixed-graphs}

A directed mixed graph $\gG = (V, \sD, \sB)$ consists of a finite vertex set
$V$, a set $\sD \subseteq V \times
V$ of \emph{directed edges}, and a set $\sB \subseteq V \times V$ of
\emph{bidirected edges} that are required to be symmetric:
\[\fitdisplay{
  (V_j,V_k) \in \sB \Longleftrightarrow (V_k,V_j) \in \sB,~\text{for
    all}~V_j,V_k \in V.
}\]
It is helpful to think about the edges as relations between the
vertices and write
\[\fitdisplay{
  V_j \rdedge V_k \ingraph{\gG} \Longleftrightarrow (V_j,V_k) \in
  \sD\quad\text{and}\quad V_j \bdedge V_k \ingraph{\gG}
  \Longleftrightarrow (V_j,V_k) \in \sB.
}\]
The choice of drawing edges in $\sB$ as bidirected instead of
undirected is intentional and crucial.\footnote{This is also where the name
``directed mixed graph'' comes from
\parencite{richardsonMarkovPropertiesAcyclic2003}.} Let $\DMG(V)$
denote
the set of all such graphs. Note that loops, whether bidirected (such
as $V_j \bdedge V_j$) or
directed (such as $V_j \rdedge V_j$), are allowed. For most of this article, we
will work with graphs that contain all bidirected loops, that is, $V_j
\bdedge V_j \ingraph{\gG}$ for all $V_j \in V$.\footnote{Loosely
  speaking, this means that we allow ``random innovations'' at each
  vertex in the corresponding statistical models.} Let $\DMGc(V)$
denote the collection of all such ``canonical'' graphs.

Some important subclasses of $\DMGc(V)$ include:
\begin{itemize}
\item $\BDG(V)$: the class of bidirected graphs (i.e.\ the directed
  edge set $\sD = \emptyset$);
\item $\DiG(V)$: the class of (essentially) directed graphs that contain no
  bidirected edges other than bidirected loops;
\item $\ADMGc(V)$: the class of acyclic directed mixed graphs (ADMGs),
  where by \emph{acyclic}, we mean there exist no cyclic directed
  walks like $V_j \rdedge \dots \rdedge V_j$ for any $V_j \in V$;
\item $\DAG(V) = \DiG(V) \cap \ADMGc(V)$: the class of (essentially)
  directed acyclic graphs (DAGs).
\end{itemize}

It is convenient to not actually mark out the bidirected loops for
``canonical'' graphs in $\DMGc(V)$. Indeed, this defines an isomorphism from
$\DMGc(V)$ that contains \emph{all} bidirected loops to the subclass
of $\DMG(V)$ that contains \emph{no} bidirected loops. For this
reason, we will generally not distinguish between these two subclasses
in this article.\footnote{One might ask why we do not start with graphs without
  bidirected loops in the first place. This is mainly because in some
  cases it is useful to consider graphs in which some vertices have
  bidirected loops and some do not. One example is the single-world
  intervention graphs \parencite{richardson2013single}.}

Let us introduce a new subclass of $\DMGc(V)$ that will play an
important role.

\begin{definition}
  Given $\gG \in \DMGc(V)$, the set of \emph{exogenous} vertices is
  defined as
  \[\fitdisplay{
    E = \{V_j \in V: V_k \nordedge V_j~\text{for all}~V_k \in
    V\}.
  }\]
  We say $\gG$ is \emph{exogenously bidirected} if for all $V_j, V_k \in V$ such
  that $V_j \neq V_k$, we have
  \begin{equation*}
    V_j \bdedge V_k \ingraph{\gG} \Longrightarrow V_j, V_k \in
    E.
  \end{equation*}
  Let $\UnDMGc(V)$ denote the set of all such exogenously bidirected
  graphs with vertex set $V$ and $\UnADMGc(V) = \UnDMGc(V) \cap
  \ADMGc(V)$.
\end{definition}

The semantics in this definition is simple: the exogenous vertices
have some underlying structure described by the bidirected edges, and
they influence the rest of the \emph{endogenous} vertices in a
recursive way through the directed edges.
It is
obvious that this subclass contains DAGs and bidirected graphs:
$\DAG(V) \subseteq
\UnADMGc(V)$ and $\BDG(V) \subseteq \UnADMGc(V)$. We will see shortly
that exogenously bidirected ADMGs share many good properties with DAGs and
bidirected graphs.  %

Note that a similar but different type of graphs is considered by
\textcite{kiiveriRecursiveCausalModels1984}. There, the exogenous
variables are connected by undirected edges and are required to
satisfy the global Markov property for undirected graphs, so what they
consider is a subclass of the chain graph models
\parencite{lauritzenGraphicalModelsAssociations1989,frydenbergChainGraphMarkov1990}.

\subsection{Statistical models associated with ADMGs and their relations}

\begin{table}[!t]
  \caption{Guide to the models and graph classes used below. Model
    abbreviations denote sets of distributions; graph-class symbols denote sets of graphs.}
  \label{tab:glossary}
  \small
  \begin{tabularx}{\linewidth}{@{}lXl@{}}
    \toprule
    Symbol & Meaning & Definition \\
    \midrule
    PE & Margin of the pairwise DAG expansion. & \S\ref{sec:pairw-cliq-noise} \\
    CE & Margin of the clique DAG expansion. & \S\ref{sec:pairw-cliq-noise} \\
    NE & Margin of the noise expansion, with bidirected dependence among noises. & \S\ref{sec:pairw-cliq-noise} \\
    E & Recursive equations with bidirected noise dependence; equivalent to NE. & \S\ref{sec:nonp-syst} \\
    EF & Exogenous factorization, for exogenously bidirected ADMGs. & \S\ref{sec:fact-prop} \\
    F & Factorization, for DAGs. & \S\ref{sec:fact-prop} \\
    NM & Markov constraints in every validly fixed kernel. & \S\ref{sec:nest-mark-prop} \\
    M & Markov: m-separation implies conditional independence; equivalently, ordered local Markov (LM). & \S\ref{sec:global-markov-models} \\
    A & Augmentation criterion; equivalent to M. & \S\ref{sec:augm-crit} \\
    MI & Marginal independences implied by m-separation. & \S\ref{sec:conn-mark-models} \\
    \midrule
    $\ADMGc(V)$ & Acyclic directed mixed graphs (ADMGs): no directed cycles. & \S\ref{sec:direct-mixed-graphs} \\
    $\UnADMGc(V)$ & Exogenously bidirected ADMGs: bidirected edges between distinct vertices join only exogenous vertices. & \S\ref{sec:direct-mixed-graphs} \\
    $\DAG(V)$ & Directed acyclic graphs (DAGs): ADMGs without
                bidirected edges between distinct vertices. & \S\ref{sec:direct-mixed-graphs} \\
    $\BDG(V)$ & Bidirected graphs: no directed edges. & \S\ref{sec:direct-mixed-graphs} \\
    \bottomrule
  \end{tabularx}
\end{table}

 In graphical statistical models, vertices in the graph are random
variables, and there are different ways to interpret the edges as
relationships between the variables. %
To formalize such interpretations as statistical models, it is helpful
to take the more
abstract point of view that a statistical model is a collection of
probability distributions. Let $V =
(V_1,\dots,V_d)$ be a random vector that takes values in a product
measure space $\mathbb{V} =\mathbb{V}_1 \times \dots \times
\mathbb{V}_d$. Write $\mathbb P(\mathbb V)$ for the \emph{nonparametric
  model}---the set of all probability distributions on $\mathbb V$;
see \Cref{sec:marginalization}.

Graphical statistical models associate graphs with subclasses of
$\mathbb{P}(\mathbb{V})$; in other words, they are maps from
$\ADMGc(V)$ to the power set of $\mathbb{P}(\mathbb{V})$. Let us
illustrate this by introducing some common ADMG models here:

\begin{enumerate}
\item One approach is to associate certain separating relations in
  the graph with conditional independences in the probability
  distribution. For example, for $\gG
  \in \DMGc(V)$, the \emph{Markov}
  (M) model collects all distributions $\P \in \statmodel(\mathbb{V})$
  that obey the (global)
  Markov property with respect to $\gG$: m-separation in $\gG$ (this
  will be defined in \Cref{sec:walk-algebra}) implies conditional
  independence under $\P$. This was first formally introduced by
  \textcite{richardsonMarkovPropertiesAcyclic2003} but goes back to
  investigations of (cyclic) linear structural equation models in the
  previous decade.
\item Another approach is to consider certain ``expansions'' of the
  graph that have a simpler structure. For example, the \emph{clique
    expansion} (CE) model expands every clique of bidirected edges with
  a latent variable, so the resulting graph is a DAG. The \emph{noise
    expansion} (NE) model associates each vertex in the graph with a
  latent variable that inherits all its bidirected edges, so the
  resulting graph is exogenously bidirected. See \Cref{fig:examples} below for some
  examples. Many authors take this approach implicitly without fully
  defining their model; a more explicit account is given in
  \textcite[Section 4.1]{richardsonNestedMarkovProperties2023}.
\item Alternatively, one can consider a system of equations that obey
  the local structure of the graph. The
  \emph{nonparametric system of equations} (E) model %
  collects observed laws of $V$ admitting a joint realization with
  noise variables such that
   \[\fitdisplay{
    V_j = f_j(V_{\pa(j)}, E_j),\quad\text{for all}~V_j \in V,
  }\]
  for some functions $f_1,\dots,f_d$, where $\pa(j) = \{k: V_k \rdedge
  V_j \ingraph{\gG}\}$ is the parent
  set of $V_j$ in $\gG$ and, importantly, the distribution of the
  ``noise variables'' $E = (E_1,\dots,E_d)$ is global Markov with
  respect to the bidirected component of $\gG$. This is closely related
  to the ``semi-Markovian'' causal model in \textcite[p.\
  30]{pearlCausality2009} and
  \textcite{bareinboimPearlHierarchyFoundations2022} who leave the
  distribution of $E$ unspecified.%
\end{enumerate}

\begin{figure}[h]
   \centering
  \begin{subfigure}[b]{0.30\linewidth} \centering
    \scalebox{1.4}{%
    \mathsurround=0pt\relax%
    \begin{tikzcd}[ampersand replacement=\&,row sep=small,column sep=small]
      \mathrm{PE}\arrow[d,Rightarrow] \& \\
      \mathrm{CE}\arrow[d,Rightarrow,very thick] \& \\
      \mathrm{NE}\arrow[r,Leftrightarrow,very thick]
        \arrow[d,Rightarrow,very thick] \& \mathrm{E} \\
      \mathrm{NM}\arrow[d,Rightarrow] \& \\
      \mathrm{M}\arrow[r,Leftrightarrow]\arrow[d,Rightarrow] \& \mathrm{A} \\
      \mathrm{MI} \&
    \end{tikzcd}}
    \caption{$\gG$ is an ADMG: $\gG \in \ADMGc(V)$.}
    \label{fig:relation-statmodel-admg}
  \end{subfigure} \qquad
  \begin{subfigure}[b]{0.61\linewidth} \centering
    \scalebox{1.4}{%
    \mathsurround=0pt\relax%
    \begin{tikzcd}[ampersand replacement=\&,row sep=small,column sep=small]
      \mathrm{PE}\arrow[d,Rightarrow] \& \& \& \& \& \\
      \mathrm{CE}\arrow[d,Rightarrow] \& \& \& \& \& \\
      \mathrm{NE}\arrow[r,Leftrightarrow]\arrow[d,Rightarrow] \&
      \mathrm{E}\arrow[r,Leftrightarrow,very thick] \&
      \mathrm{EF}\arrow[r,Leftrightarrow]
        \arrow[rr,Leftrightarrow,very thick,bend right=40] \&
      \mathrm{NM}\arrow[r,Leftrightarrow] \&
      \mathrm{M}\arrow[r,Leftrightarrow] \& \mathrm{A} \\
      \mathrm{MI} \& \& \& \& \&
    \end{tikzcd}}
     \caption{$\gG$ is exogenously bidirected: $\gG \in \UnADMGc(V)$.}
    \label{fig:relation-statmodel-unconfounded}
  \end{subfigure}

   \begin{subfigure}[b]{0.94\linewidth} \centering
    \scalebox{1.4}{%
    \mathsurround=0pt\relax%
    \begin{tikzcd}[ampersand replacement=\&,row sep=small,column sep=small]
      \mathrm{PE}\arrow[r,Leftrightarrow]\arrow[d,Rightarrow] \&
      \mathrm{CE}\arrow[r,Leftrightarrow] \&
      \mathrm{NE}\arrow[r,Leftrightarrow] \&
      \mathrm{E}\arrow[r,Leftrightarrow] \&
      \mathrm{EF}\arrow[r,Leftrightarrow] \&
      \mathrm{F}\arrow[r,Leftrightarrow] \&
      \mathrm{NM}\arrow[r,Leftrightarrow] \&
      \mathrm{M}\arrow[r,Leftrightarrow] \& \mathrm{A} \\
      \mathrm{MI} \& \& \& \& \& \& \& \&
    \end{tikzcd}}
     \caption{$\gG$ is a DAG: $\gG \in \DAG(V)$.}
    \label{fig:relation-statmodel-dag}
  \end{subfigure}

   \begin{subfigure}[b]{0.80\linewidth} \centering
    \scalebox{1.4}{%
    \mathsurround=0pt\relax%
    \begin{tikzcd}[ampersand replacement=\&,row sep=small,column sep=small]
      \mathrm{PE}\arrow[d,Rightarrow] \& \& \& \& \& \& \\
      \mathrm{CE}\arrow[d,Rightarrow] \& \& \& \& \& \& \\
      \mathrm{NE}\arrow[r,Leftrightarrow] \&
      \mathrm{E}\arrow[r,Leftrightarrow] \&
      \mathrm{EF}\arrow[r,Leftrightarrow] \&
      \mathrm{M}\arrow[r,Leftrightarrow] \&
      \mathrm{A}\arrow[r,Leftrightarrow] \&
      \mathrm{MI}\arrow[r,Leftrightarrow] \& \mathrm{NM}
    \end{tikzcd}}

 \caption{$\gG$ is a bidirected graph: $\gG \in \BDG(V)$.}
    \label{fig:relation-statmodel-bdg}
  \end{subfigure}
  \caption{Relations between statistical models associated with ADMGs;
    definitions are in \Cref{sec:stat-models-assoc} and abbreviations
    in \Cref{tab:glossary}; A denotes the augmentation model
    (\Cref{sec:augm-crit}). %
    Throughout this paper, NM is considered only for discrete (finite or
    countable) state spaces; the other models can have more general
    state spaces. Bold arrows mark new results established in this
    paper; other arrows are known, definitional, or implied by
    others.}
  \label{fig:relation-statmodel}
\end{figure}

We will formally define the above models and some other
interpretations of ADMGs in \Cref{sec:stat-models-assoc}.
The next theorem summarizes these relations, many of which are
already known.

 \begin{theorem} \label{thm:relation-statmodel}
  The relations in \Cref{fig:relation-statmodel} hold for all $\gG$ in
  the corresponding classes of graphs, where
  $\Rightarrow$ ($\Leftrightarrow$) should be interpreted as
  $\subseteq$ ($=$) for the corresponding models
  with the same state space $\mathbb{V}$. Moreover,
  all $\Rightarrow$ in \Cref{fig:relation-statmodel} are strict in the
  sense that the reverse implications are not true for some
  graphs (and non-trivial state spaces) in the corresponding subclasses.%
\end{theorem}

Two substantive results proved in this paper deserve
some explanation. First, the E/NE model is nested
Markov (\Cref{thm:admg-ne-nm} below).\footnote{To avoid overcomplicating the results with
  measure-theoretic details, we only consider nested Markov
  models with finite or countable state spaces throughout this paper.}
\textcite[Theorem 46]{richardsonNestedMarkovProperties2023} established
that margins of DAG models are nested Markov, giving
$\text{CE}\Rightarrow\text{NM}$ in our terminology. The E/NE model
can be strictly larger than CE, so that result does not imply
$\text{E/NE}\Rightarrow\text{NM}$. Our proof relies on a natural
definition of a causal Markov model directly on ADMGs; it is outlined in
\Cref{sec:causal-model}.
Second, exogenous factorization is equivalent to the Markov
property on exogenously bidirected graphs:
$\text{EF}\Leftrightarrow\text{M}$. This is a new
factorization theorem that shows the equivalence of E/NE, EF, M, A,
and NM on this graph class.

Although we have not introduced many statistical models in
\Cref{fig:relation-statmodel} yet, some high-level observations can
already be made:
\begin{enumerate}
\item Exogenously bidirected ADMGs share the equivalences of statistical models
  that are found for DAGs and bidirected graphs. For example,
  $\text{E} \Leftrightarrow \text{M}$ is true for exogenously bidirected ADMGs
  (and thus DAGs and bidirected graphs) but not all ADMGs. For this
  reason, exogenously bidirected ADMGs may be considered as the natural
  generalization of DAGs and bidirected graphs.
\item A general ADMG is associated with many ``tiers'' of
  non-equivalent statistical models. Thus, the Equivalence argument
  does not give a natural definition of a statistical model for all
  ADMGs.
\end{enumerate}

\subsection{Graph expansion and complete models}
\label{sec:agnostic-models}

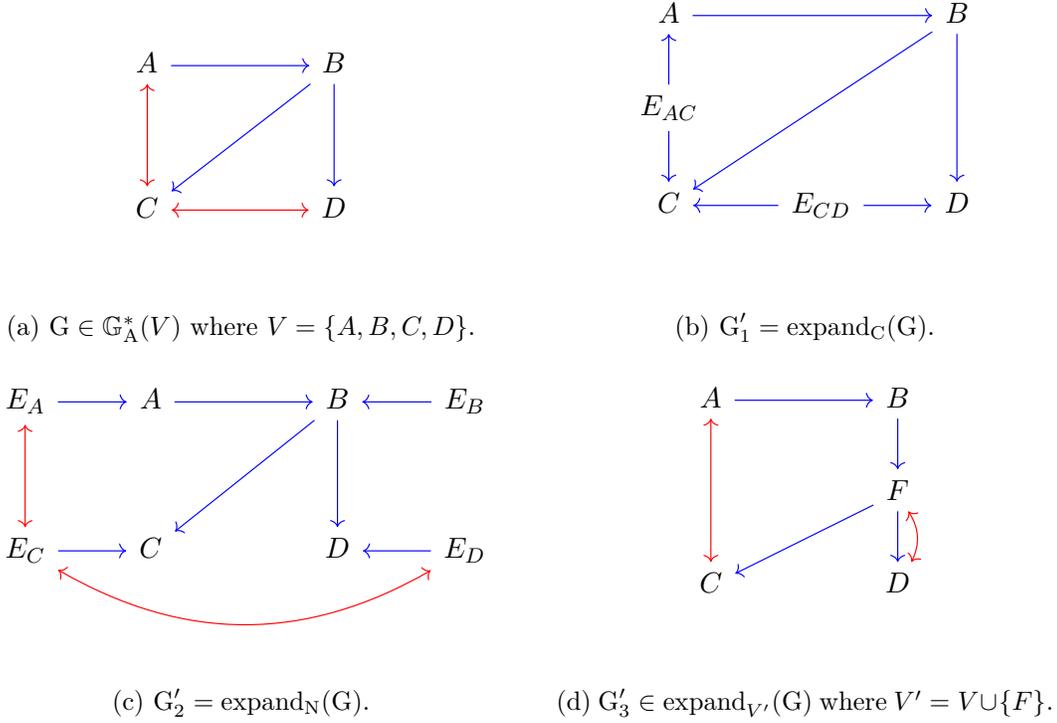
\begin{figure}[t]
  \centering
  \begin{subfigure}[b]{0.45\linewidth} \centering
    \begin{tikzpicture}[
      font=\fontsize{10}{12}\selectfont, inner sep=3pt,
      >={Computer Modern Rightarrow[length=0pt 6.2]}, line width=0.4pt]
      \node (A) at (0,1.8) {$A$};
      \node (B) at (2.3,1.8) {$B$};
      \node (C) at (0,0) {$C$};
      \node (D) at (2.3,0) {$D$};
      \draw[->,blue] (A) -- (B);
      \draw[->,blue] (B) -- (C);
      \draw[->,blue] (B) -- (D);
      \draw[<->,red] (A) -- (C);
      \draw[<->,red] (C) -- (D);
    \end{tikzpicture}
    \caption{$\gG \in \ADMGc(V)$ where $V = \{A,B,C,D\}$.}
    \label{fig:examples-G}
  \end{subfigure}
  \qquad
  \begin{subfigure}[b]{0.45\linewidth}\centering
    \begin{tikzpicture}[
      font=\fontsize{10}{12}\selectfont, inner sep=3pt,
      >={Computer Modern Rightarrow[length=0pt 6.2]}, line width=0.4pt]
      \node (A) at (0,2.3) {$A$};
      \node (B) at (3.5,2.3) {$B$};
      \node (C) at (0,0) {$C$};
      \node (D) at (3.5,0) {$D$};
      \node (EAC) at (0,1.15) {$E_{AC}$};
      \node (ECD) at (1.75,0) {$E_{CD}$};
      \draw[->,blue] (A) -- (B);
      \draw[->,blue] (B) -- (C);
      \draw[->,blue] (B) -- (D);
      \draw[->,blue] (EAC) -- (A);
      \draw[->,blue] (EAC) -- (C);
      \draw[->,blue] (ECD) -- (C);
      \draw[->,blue] (ECD) -- (D);
    \end{tikzpicture}
    \caption{$\gG_1' = \expandC(\gG)$.}
    \label{fig:examples-expandC}
  \end{subfigure}

  \begin{subfigure}[b]{0.45\linewidth} \centering
    \begin{tikzpicture}[
      font=\fontsize{10}{12}\selectfont, inner sep=3pt,
      >={Computer Modern Rightarrow[length=0pt 6.2]}, line width=0.4pt]
      \node (EA) at (0,1.8) {$E_A$};
      \node (A) at (1.65,1.8) {$A$};
      \node (B) at (3.95,1.8) {$B$};
      \node (EB) at (5.5,1.8) {$E_B$};
      \node (EC) at (0,0) {$E_C$};
      \node (C) at (1.65,0) {$C$};
      \node (D) at (3.95,0) {$D$};
      \node (ED) at (5.5,0) {$E_D$};
      \draw[->,blue] (A) -- (B);
      \draw[->,blue] (B) -- (C);
      \draw[->,blue] (B) -- (D);
      \draw[->,blue] (EA) -- (A);
      \draw[->,blue] (EB) -- (B);
      \draw[->,blue] (EC) -- (C);
      \draw[->,blue] (ED) -- (D);
      \draw[<->,red] (EA) -- (EC);
      \draw[<->,red] (EC) to[bend right=30] (ED);
    \end{tikzpicture}
    \caption{$\gG_2' = \expandN(\gG)$.}
    \label{fig:examples-expandE}
  \end{subfigure}
  \qquad
  \begin{subfigure}[b]{0.45\linewidth} \centering
    \begin{tikzpicture}[
      font=\fontsize{10}{12}\selectfont, inner sep=3pt,
      >={Computer Modern Rightarrow[length=0pt 6.2]}, line width=0.4pt]
      \node (A) at (0,2.3) {$A$};
      \node (B) at (2.3,2.3) {$B$};
      \node (C) at (0,0) {$C$};
      \node (D) at (2.3,0) {$D$};
      \node (F) at (2.3,1.15) {$F$};
      \draw[->,blue] (A) -- (B);
      \draw[->,blue] (B) -- (F);
      \draw[->,blue] (F) -- (C);
      \draw[->,blue] (F) -- (D);
      \draw[<->,red] (A) -- (C);
      \draw[<->,red] (F) to[bend left=30] (D);
    \end{tikzpicture}
    \caption{$\gG_3' \in \expand_{V'}(\gG)$ where $V' = V \cup
      \{F\}$.}
    \label{fig:examples-expand}
  \end{subfigure}
  \caption{Examples of graph expansion (all bidirected loops are
    omitted).}
  \label{fig:examples}
\end{figure}

We will now turn to the Completion argument and define what we mean by
``complete''. To this end, let $\marg_V$ denote the (overloaded)
``marginalization'' operator on ADMGs (that maps $\ADMGc(V')$ to
$\ADMGc(V)$ for some $V' \supseteq V$) and probability distributions
(acting by pushforward on all probability laws and mapping
$\statmodel(\mathbb V')$ to $\statmodel(\mathbb V)$); these will be
formally defined in \Cref{sec:marginalization}. Let $\expand_{V'}$ denote the
pre-image of $\marg_V$, that is,
\[\fitdisplay{
  \expand_{V'}(\gG) = \bigl \{\gG' \in \ADMGc(V'): \marg_V(\gG') = \gG
  \bigr\}.
}\]

\begin{definition} \label{def:agnostic}
  For every finite vertex set $V$, let $\mathcal{G}_0(V)\subseteq
  \ADMGc(V)$ be a given subclass of (canonical) ADMGs. A collection of
  statistical models $\statmodel(\gG,\mathbb{V})$ for different ADMGs
  $\gG$ is said to be \emph{complete} (with respect to expansions in
  $\mathcal{G}_0$) if it is equal to the union of the $V$-marginals  of
  all $\mathcal{G}_0$-``expanded'' models, that is,
  \begin{equation}\label{eq:agnosticism}
    \begin{split}
      \statmodel (\gG,\mathbb{V} )
       ={}& \bigcup_{V'\supseteq V}\  \bigcup_{\gG'\in
           \expand_{V'}(\gG)\cap\mathcal{G}_0(V')} \marg_V\big(\statmodel
           ( \gG',\mathbb{V}\times[0,1]^{|V'\setminus V|} )\big) ,
    \end{split}
   \end{equation}
  where the first union is over finite vertex sets $V'\supseteq V$ and
  the added variables take values in $[0,1]$. We will call the right-hand side the
  \emph{completion} of the specified base models on
  $\mathcal{G}_0$.\footnote{Although the NM model is
  defined below only for finite or countable state spaces, in its completion
  comparison, expanded DAGs and exogenously bidirected graphs with added
  $[0,1]$-valued vertices use M as their base model; M and NM agree on
  these graph classes when the state spaces of all modeled vertices are finite or countable.}
\end{definition}

The trivial expansion $V' = V$ contributes in \eqref{eq:agnosticism}
only when $\gG$ belongs to the base class. Outside that class, completeness
requires every model distribution to admit a representation on an
allowed graph expansion.

Completeness is a desirable property because it allows us to be
agnostic about the particular graph expansion (latent variable
``explanation'' of the data). In other words, a complete ADMG model does not try
to tell us \emph{why} two variables are related. For instance, when
the ADMG is interpreted as a causal model, a directed edge is usually
interpreted as a direct causal effect not through other variables in
the graph. It is entirely possible that such a direct causal effect is
mediated through one or multiple latent variables, but that is not
part of a complete model.

Equation \eqref{eq:agnosticism} can be viewed as a way to extend
some ``base models''---statistical models for a smaller class of graphs---to
a larger class of graphs. This heuristic can be widely used in the
literature to interpret ADMGs; for example, \textcite[p.\
76]{pearlCausality2009} writes
``... especially true in semi-Markovian models (i.e., DAGs involving
unmeasured variables)''. This intuitive ``latent DAG'' interpretation is
formalized in \textcite[Section 4.1]{richardsonNestedMarkovProperties2023}
who use Markov models on DAGs as the base models.
\Cref{thm:statmodelE-agnostic} below further shows that
this latent DAG interpretation is equivalent to the clique expansion
(CE) model defined in \Cref{sec:pairw-cliq-noise}. Thus, rather than
using \eqref{eq:agnosticism} as an abstract definition of a statistical
model, we present it as a completeness property.

Besides DAG expansions ($\mathcal{G}_0(V) = \DAG(V)$), we will also
consider exogenously bidirected expansions ($\mathcal{G}_0(V) =
\UnADMGc(V)$) as the base graphs. The next Theorem
summarizes our second main result. %

\begin{figure}[t]
  \centering
  \resizebox{\linewidth}{!}{%
  \mathsurround=0pt\relax%
  \begin{tikzpicture}[
    font=\small,
    model/.style={inner sep=4pt},
    equiv/.style={draw=black!65,rounded corners=2pt,inner xsep=8pt,inner ysep=6pt,align=center},
    implication/.style={double,double distance=.9pt,-{Implies[length=4pt,width=5pt]},shorten >=2pt,shorten <=2pt},
    completion/.style={-{Stealth[length=5pt]},draw=red!75!black,line width=1.1pt,shorten >=3pt,shorten <=3pt}
  ]
    \node[align=center] at (0,1) {General ADMGs\\$\gG\in\ADMGc(V)$};
    \node[align=center] at (4.6,1) {Exogenously bidirected ADMGs\\$\gG\in\UnADMGc(V)$};
    \node[align=center] at (9.2,1) {DAGs\\$\gG\in\DAG(V)$};
    \node[model] (gpe) at (0,0) {$\mathrm{PE}$};
    \node[model] (gce) at (0,-.85) {$\mathrm{CE}$};
    \node[equiv] (gne) at (0,-2) {$\mathrm{NE},\ \mathrm E$};
    \node[model] (gnm) at (0,-3) {$\mathrm{NM}$};
    \node[equiv] (gm) at (0,-3.9) {$\mathrm M,\ \mathrm A$};
    \node[model] (gmi) at (0,-5) {$\mathrm{MI}$};
    \foreach \a/\b in {gpe/gce,gce/gne,gne/gnm,gnm/gm,gm/gmi}
      \draw[implication] (\a) -- (\b);
    \node[model] (epe) at (4.6,0) {$\mathrm{PE}$};
    \node[model] (ece) at (4.6,-.85) {$\mathrm{CE}$};
    \node[equiv] (ene) at (4.6,-2) {$\mathrm{NE},\ \mathrm E,\ \mathrm{EF}$\\[2pt]$\mathrm{NM},\ \mathrm M,\ \mathrm A$};
    \node[model] (emi) at (4.6,-5) {$\mathrm{MI}$};
    \draw[implication] (epe) -- (ece);
    \draw[implication] (ece) -- (ene);
    \draw[implication] (ene) -- (emi);
    \node[equiv] (dag) at (9.2,-.85) {$\mathrm{PE},\ \mathrm{CE},\ \mathrm{NE},\ \mathrm E$\\[2pt]$\mathrm{EF},\ \mathrm F,\ \mathrm{NM},\ \mathrm M,\ \mathrm A$};
    \node[model] (dmi) at (9.2,-5) {$\mathrm{MI}$};
    \draw[implication] (dag) -- (dmi);
    \draw[completion] (dag.west) -- (ece.east);
    \draw[completion] (ece.west) -- (gce.east);
    \draw[completion] (ene.west) -- (gne.east);
    \draw[completion] (dmi.west) -- (emi.east);
    \draw[completion] (emi.west) -- (gmi.east);
  \end{tikzpicture}}
   \caption{Completion of ADMG models. Red arrows show the completions in \eqref{eq:agnosticism}.}
  \label{fig:completion}
\end{figure}
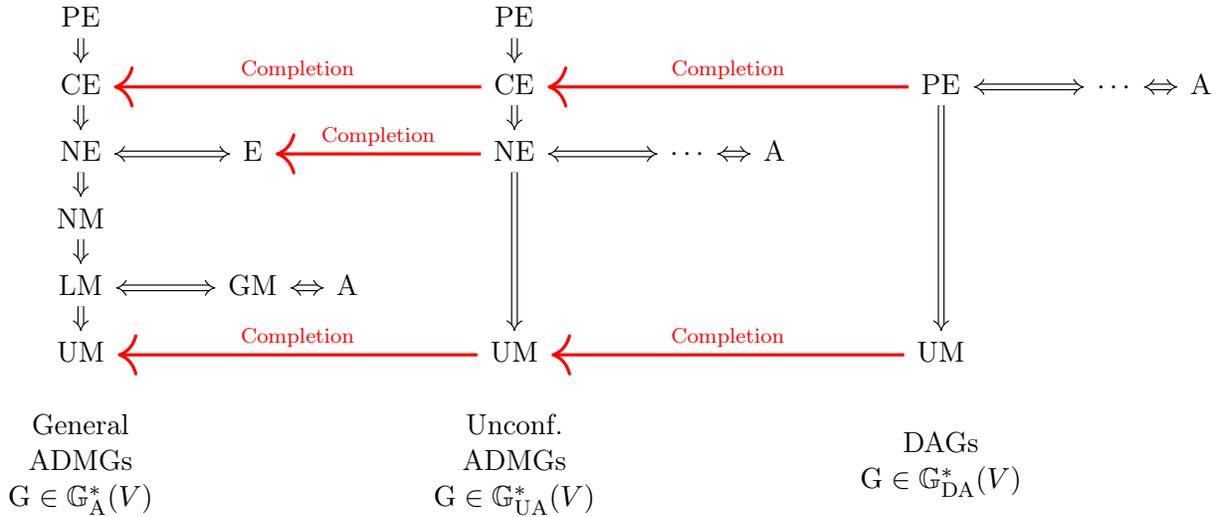

\begin{theorem} \label{thm:statmodelE-agnostic}
  Consider all the ADMG models in
  \Cref{fig:relation-statmodel-admg}. We have the following:
  \begin{enumerate}
  \item Taking $\mathcal{G}_0(V) = \UnADMGc(V)$ in
    \Cref{def:agnostic}, only the CE, E, NE, and MI models are
    complete.
  \item Taking $\mathcal{G}_0(V) = \DAG(V)$ in
    \Cref{def:agnostic}, only the CE and MI models
    are complete.
  \end{enumerate}
\end{theorem}

\Cref{fig:completion} visualizes the results in
\Cref{thm:relation-statmodel,thm:statmodelE-agnostic}. It shows that
if we use both the Equivalence and the Completion arguments, there are
two good candidates for an ADMG model:
\begin{enumerate}
\item The clique expansion (CE) model, which is complete with respect
  to DAG expansions.
\item The noise expansion (NE) model (or equivalently the E model),
  which is complete with respect to exogenously bidirected expansions.
\end{enumerate}
The choice between these two models rests on whether one finds it more
attractive to use DAGs or exogenously bidirected graphs as the base class. In
the author's opinion, the latter is more natural because no
explanation is needed for the bidirected edges. Intuitively, a
bidirected edge means that two variables are correlated in an
exogenous way, possibly due to one or multiple latent common
causes. However, the nature of that exogenous correlation is not part
of the NE model. We will return to more discussion on this in
\Cref{sec:discussion}.

Let us take a moment to illustrate the definition of
completeness. Consider
the graphs in \Cref{fig:examples}, using Markov models on
exogenously bidirected graphs as the base models so
$\mathcal{G}_0(V)=\UnADMGc(V)$.
The graph $\gG \in
\ADMGc(V)$ for $V = \{A, B, C, D\}$ in \Cref{fig:examples-G} is not
exogenously bidirected, but after the
``clique'' (\Cref{fig:examples-expandC}) or
``noise'' expansion (\Cref{fig:examples-expandE}), it becomes an
exogenously bidirected graph. The remark after \Cref{thm:relation-statmodel}
suggests that the nonparametric system of equations is a natural
statistical model associated with such graphs.
\Cref{fig:examples-expand} shows another possible
expansion of $\gG$ that involves a latent variable $F$, but the
expanded graph is not exogenously bidirected because of $A \bdedge C \ldedge F$ and $B
\rdedge F \bdedge D$ (one can further expand the bidirected edges to
make the graph exogenously bidirected). For the model for
$\gG$ to be complete with respect to exogenously bidirected graphs,
it must contain the $V$-marginals of $\statmodelGM(\gG_1')$,
$\statmodelGM(\gG_2')$, and every other exogenously bidirected
expansion of $\gG$; \Cref{thm:statmodelE-agnostic} shows that this
union is precisely $\statmodelNE(\gG)=\statmodelE(\gG)$.%

The rest of this article is organized as follows. In
\Cref{sec:notation-terminology}, we introduce some basic notation and
terminology for graphical statistical models. In
\Cref{sec:stat-models-assoc}, we formally define the statistical models
that appear in \Cref{thm:relation-statmodel}. In
\Cref{sec:causal-model}, we outline a proof of the assertion that the
nonparametric equation system is
nested Markov by building a theory
for causality based on ADMGs. In \Cref{sec:discussion}, we give some
further remarks on causal models associated with ADMGs. Technical
proofs can be found in \Cref{sec:technical-proofs}.

\section{Basic notation and terminology}
\label{sec:notation-terminology}

\subsection{Walk algebra}
\label{sec:walk-algebra}

We adopt the notation and terminology in
\textcite{zhaoMatrixAlgebraGraphical2024} to describe the walk algebra
generated by directed mixed graphs. For $V_j, V_k \in V$, we say $w$
is a \emph{walk} from $V_j$ to $V_k$
if it is a sequence of connecting edges (edge directions are ignored
when deciding connection), its first edge starts at $V_j$, and its
last edge ends at $V_k$. We say a walk is \emph{simple} if its
end-points appear only once, and we say a walk is a \emph{path} if all
vertices in it appear only once. We say a walk is \emph{blocked} by $L
\subseteq V$ if
\begin{enumerate}
\item $w$ contains a collider $V_l$ (so part of $w$ looks like
  $\rdedge V_l \ldedge$, $\bdedge V_l \ldedge$, $\rdedge V_l
  \bdedge$, or $\bdedge V_l \bdedge$) such that $V_l \not \in L$; \text{or}
\item $w$ contains a non-collider $V_l$ such that $V_l \in L$.
\end{enumerate}
This is slightly different from (but in canonical ADMGs equivalent to)
the path-based notion of blocking usually used in the literature, which
requires that neither a collider $V_l$ nor any of its descendants is in
$L$. The walk-based blocking and path-based blocking define the same
m-separation criterion for canonical directed mixed graphs. The
walk-based definition of blocking is easier to work with because it
does not depend on the ambient graph. See
\textcite{zhaoMatrixAlgebraGraphical2024} for further discussion.

If $\{V_j\}$, $\{V_k\}$, and $L$ are disjoint and there exists an
unblocked walk from $V_j$ to $V_k$ given $L$, we say $V_j$ is
\emph{m-connected} to $V_k$ given $L$ and write $V_j
\mconn V_k \mid L \ingraph{\gG}$; the half arrowheads mean the walk
can end with or without arrowheads on both sides, and the asterisk is
a wildcard character to indicate that the walk may have any number of
colliders. If no such walk
exists, we say $V_j$ and $V_k$ are \emph{m-separated} given $L$ in
$\gG$ and write $\textnot V_j \mconn V_k \mid L \ingraph{\gG}$. This
definition naturally extends to sets of vertices: for disjoint $J,
K, L \subset V$, we write
\[\fitdisplay{
  J \mconn K \mid L \ingraph{\gG} \Longleftrightarrow
  V_j \mconn V_k \mid L \ingraph{\gG}~\text{for some}~V_j \in
  J, V_k \in K.
}\]

We now introduce some special types of walks and associated concepts
that play important roles in the theory of ADMG models:
\begin{enumerate}
\item $\rdedge$ and $\bdedge$\,: these are the basic edges that generate
  the walk algebra. We write $\pa_{\gG}(V_j) = \{V_k \in V: V_k \rdedge V_j\}$
  as the \emph{parents} of $V_j$ and $\ch_{\gG}(V_j) = \{V_k \in V: V_j
  \rdedge V_k\}$ as the \emph{children} of $V_j$. When it is more
  convenient to work with indices of the variables, we often
  use the notation $\pa_{\gG}(j) = \{k \in [d]: V_k \rdedge V_j\}$ and
  likewise for $\ch_{\gG}(j)$. We sometimes omit the graph $\gG$ in
  the subscript when it is clear from the context.
\item $\rdpath$\,: this means a \emph{(right-)directed walk} that
  consists of one or more $\rdedge$. We write $\an(V_j) = \{V_k \in V:
  V_k \rdpath V_j\}$
   as the \emph{ancestors} of $V_j$ and $\de(V_j) = \{V_k \in V: V_j
  \rdpath V_k\}$ as the \emph{descendants} of $V_j$. In an ADMG,
  neither set contains $V_j$ itself. The corresponding
  indices are denoted as $\an(j)$ and $\de(j)$, respectively. We say a
  subset of vertices $J \subseteq V$ is \emph{ancestral} in $\gG$ if
  it contains all its ancestors, that is, $\an(J) \subseteq J$ where
  \[\fitdisplay{
    \an(J) = \bigcup_{V_j \in J} \an(V_j) = \{V_k \in V: V_k \rdpath J
    \ingraph{\gG}\}.
  }\]
  This concept is useful because the ancestral marginal of an ADMG is
  simply its induced subgraph, that is, if $J$ is ancestral, then
  $\marg_J(\gG) = (J, \sD \cap (J \times J), \sB \cap (J \times J))$
  for $\gG = (V, \sD, \sB)$; see \Cref{sec:marginalization} below for
  the definition of graph marginalization.
\item $\mconnarc$\,: this means an \emph{arc}, a walk with no colliders.
\item $\confarc$\,: these are all arcs that are not $\rdpath$ or
  $\ldpath$ (consisting of one or more $\ldedge$). When a walk like
  $\confarc$ is a path, we call it a \emph{confounding arc}.
\item $\samedist$\,: this means a walk consisting of one or more
  $\bdedge$. The set $\dis_{\gG}(V_j) = \{V_k \in V: V_k \samedist
  V_j \ingraph{\gG}\}$ is called the \emph{district} of $V_j$ in
  $\gG$.\footnote{This terminology is due to
    \textcite{richardsonMarkovPropertiesAcyclic2003}. The same concept
    is called \emph{c-component} in
    \textcite{tianTestableImplicationsCausal2002a}.}
\item $\colliderconn$\,: this means a \emph{collider-connected walk} in
  which all non-endpoints are colliders. The set \[\fitdisplay{\mb_{\gG}(V_j) =
  \{V_k \in V\setminus\{V_j\}: V_k \colliderconn V_j \ingraph{\gG}\}}\] is called the
  \emph{Markov boundary} of $V_j$ in $\gG$.\footnote{This is closely
    related to the concept of Markov blanket and Markov boundary
    (minimal Markov blanket) in
    conditional independence models; see
    \textcite{pearlProbabilisticReasoningIntelligent1988}.} The
  corresponding index set is denoted as $\mb_{\gG}(j)$.
\item $\markovbgd$\,: this is a collider-connected walk that ends
  with an arrowhead. The set \[\fitdisplay{\mbg_{\gG}(V_j) =
  \{V_k \in V\setminus\{V_j\}: V_k \markovbgd V_j \ingraph{\gG}\}}\] is called the
  \emph{Markov background} of $V_j$ in
  $\gG$.\footnote{When $V_j$ has no children in $\gG$, it is obvious
    that $\mb_{\gG}(V_j) = \mbg_{\gG}(V_j)$. For this reason,
    \textcite{richardsonNestedMarkovProperties2023} also refer to
    $\mbg_{\gG}(V_j)$ as the Markov blanket/boundary of
    $V_j$. However, this terminology is confusing when $V_j$ is not
    childless and thus avoided here.} The
  corresponding index set is denoted as $\mbg_{\gG}(j)$.
\item $\mconn$\,: this is a walk consisting of one or more arcs, which
  is simply any walk in the graph.
\end{enumerate}
A formal definition of these and some other important types of walks can be
found in \textcite{zhaoMatrixAlgebraGraphical2024}.

\subsection{Marginalization}
\label{sec:marginalization}

As our argument rests on considering latent variable expansions of
graphical models, let us take some care to
define the related concepts. Consider an ADMG $\gG \in
\ADMGc(V)$. Each vertex $V_j$ is a random variable taking values in a
nonempty standard Borel space $\mathbb{V}_j$. Equip $\mathbb{V} =
\prod_{j=1}^{d} \mathbb{V}_j$ with the product $\sigma$-algebra, and
let $\mathbb{P}(\mathbb{V})$ denote the set of all probability
distributions on this space.

The marginalization operator can act on product spaces,
probability distributions, and graphs. For any subset $\sJ
\subseteq [d]$ and $J = V_{\sJ} \subseteq V$, denote the $J$-marginal
of $\mathbb{V}$ as
\[\fitdisplay{
  \marg_{J}(\mathbb{V}) = \mathbb{V}_{\sJ} = \prod_{j \in \sJ}
  \mathbb{V}_j.
}\]
Further, let $\marg_J(\P) \in \statmodel(\mathbb{V}_{\sJ})$ denote the marginal distribution of $J$
when the joint distribution of $V$ is $\P$, that is, the pushforward
of $\P\in\mathbb P(\mathbb V)$ under the coordinate
projection to $J$. Finally, for an ADMG $\gG \in
\ADMGc(V)$, its $J$-marginal is defined as its image under the map
\[\fitdisplay{\begin{aligned}
  \marg_J: \ADMGc(V) &\to \ADMGc(J), \\
  \gG &\mapsto \gG',
\end{aligned}}\]
where $\gG'$ is defined by the following equivalences for all $V_j,V_k
\in J$ such that $V_j \neq V_k$:
\[\fitdisplay{\begin{aligned}  %
  V_j \rdedge V_k \ingraph{\gG'}&\Longleftrightarrow
                                  P[V_j \rdpath V_k \mid J
                                  \ingraph{\gG}] \neq \emptyset, \\ %
  V_j \bdedge V_k \ingraph{\gG'}&\Longleftrightarrow
                                  P[V_j \confarc V_k \mid J \ingraph{\gG}]
                                  \neq
                                  \emptyset. %
\end{aligned}}\]
Here, $P$ means the set of paths, so $P[V_j \rdpath V_k \mid J
\ingraph{\gG}]$ is the set of all
directed paths from $V_j$ to $V_k$ in $\gG$ with no
non-endpoints in $J$, and $P[V_j \confarc V_k \mid J \ingraph{\gG}]$
is the set of all confounding arcs (paths with no collider and
two end-point arrowheads) from $V_j$ to $V_k$ in $\gG$ with no
non-endpoints in $J$. It can be shown that the order of
graph marginalization does not matter. Marginalization
of directed mixed graphs is often referred to as ``latent projection''
in the literature and was first considered by
\textcite{vermaEquivalenceSynthesisCausal1990}. The reader is invited
to check that the graphs in
\Cref{fig:examples-expandC,fig:examples-expandE,fig:examples-expand}
all marginalize to the graph in \Cref{fig:examples-G}.

\subsection{Conditional independence}
\label{sec:cond-indep}

As a notational convention, we use sans serif font $\P$ to denote a
probability distribution and $\p$ to denote a probability mass function when the state space is finite or countable. We use bold font
$\mathbb{P}$ to denote a collection of probability distributions.

Intuitively, a graphical statistical model imposes
algebraic (and semi-algebraic) constraints on probability
distributions according to certain structures in the graph. Perhaps
the simplest form of algebraic constraints on
probability distributions is conditional independence: for disjoint
subsets $V_{\sJ}, V_{\sK}, V_{\sL}$ of $V$, define
\[\fitdisplay{\begin{aligned}
 &V_{\sJ}\independent V_{\sK}\mid V_{\sL}\underdist{\P} \quad\Longleftrightarrow\quad
 \P(dv_{\sJ}\mid v_{\sK},v_{\sL})
 =\P(dv_{\sJ}\mid v_{\sL}),
\end{aligned}}\]
where equality is almost sure under the marginal law of
$(V_{\sK},V_{\sL})$. So given $V_{\sL}$, the value of $V_{\sK}$ is irrelevant to the
conditional distribution of $V_{\sJ}$.
Conditional independence satisfies a number of ``graphoid axioms''
that bear a close relation to graph separation; see
\textcite{pearlProbabilisticReasoningIntelligent1988,lauritzenGraphicalModels1996}.

\section{Statistical models associated with directed mixed graphs}
\label{sec:stat-models-assoc}

\subsection{Markov (M) property}
\label{sec:global-markov-models}

The Markov model, in its global formulation, assumes that every m-separation in the
graph implies a conditional independence in the probability
distribution.

\begin{definition}
  The \emph{Markov model} with respect to $\gG \in \DMGc(V)$ is defined
  as
  \[\fitdisplay{\begin{aligned}
    \statmodelGM (\gG, \mathbb{V})
    = \{\P \in \statmodel (\mathbb{V}):&
       \textnot J \mconn
       K \mid L \ingraph{\gG} \Longrightarrow J \independent K \mid L
      \underdist{\P} \\
    &\text{for all disjoint}~J,K,L \subset V \}.
  \end{aligned}}\]
\end{definition}

The Markov model takes simpler forms in some subclasses of
$\DMGc(V)$. When the graph is directed (so $\gG \in \DiG(V)$), walks
must consist of
directed edges (if we ignore bidirected loops) so $\mconn$ can be
written as $\dconn$ (where $d$ means the walk consists of directed
edges only). When the graph is bidirected (so $\gG \in \BDG(V)$),
walks must consist of bidirected edges and $\mconn$ can be written as
$\samedist$ (meaning one or more bidirected edges).

For ADMGs, the above global formulation can be reduced to a recursive,
local formulation. Given $\gG \in \ADMGc(V)$, we say a strict order $\prec$ on
the vertex set $V$ is a \emph{topological order} of $\gG$ if
\[\fitdisplay{
  V_k \rdedge V_j \ingraph{\gG} \Longrightarrow V_k \prec
  V_j~\text{for all}~ V_j, V_k \in V.
}\]
An ADMG may have multiple topological orders. Let
$\text{pre}_{\prec}(V_j) = \{V_k \in V: V_k \prec V_j\}$ collect all
vertices before $V_j$ in the order $\prec$.
Recall that the \emph{Markov boundary} of $V_j \in V$ in $\gG \in
\ADMGc(V)$ is defined as all vertices that can be connected to $V_j$
via colliders:
\[
  \fitdisplay{
  \mb_{\gG}(V_j) = \{V_k \in V\setminus\{V_j\} : V_k \colliderconn V_j \ingraph{\gG}\}.
}\]
If an ancestral set $K \subseteq V$ contains $V_j$ ($V_j \in K$) but
not any children of $V_j$ ($V_j \nordedge K \ingraph{\gG}$), the
Markov boundary of $V_j$ in $K$ is defined as
\[
  \mb_{\gG}(V_j, K) =
  \{V_k \in K \setminus \{V_j\}: V_k \markovbgd V_j \ingraph{\gG}\} =
  \mb_{\marg_K(\gG)}(V_j) = \mbg_{\marg_K(\gG)}(V_j).
\]
The reader is invited to verify the last two equalities.

The next definition is due to \textcite[p.\ 151]{richardsonMarkovPropertiesAcyclic2003}.

\begin{definition}
  The \emph{ordered local Markov model} with respect to $\gG \in \ADMGc(V)$
  and a topological order $\prec$ of $\gG$ is defined as
  \[\fitdisplay{\begin{aligned}
    \statmodelLM(\gG, \prec, \mathbb{V}) = \Big\{\P \in
    \mathbb P (\mathbb{V}): &V_j
                              \independent K \setminus ( \mb_{\gG}(V_j,K)\cup\{ V_j\})  \mid
                              \mb_{\gG}(V_j,K)
                              \underdist{\P} \\
                            &\text{for all}~V_j~\text{and
                              ancestral}~K~\text{such that}~V_j \in K
                              \subseteq
                              \text{pre}_{\prec}(V_j)\cup\{V_j\}   \Big\}.
  \end{aligned}}\]
\end{definition}

It can be shown that the model $\statmodelLM(\gG, \prec, \mathbb{V})$
actually does not depend on which topological order $\prec$ is
used, and it is the same as the (global) Markov model $\statmodelGM(\gG,
\mathbb{V})$ and the augmentation model in \Cref{sec:augm-crit}
\parencite[Theorem 2]{richardsonMarkovPropertiesAcyclic2003}.

When $\gG$ is a DAG (i.e.\ $\gG \in \DAG(V)$), the Markov boundary of
$V_j \in V$ reduces to
\[\fitdisplay{
  \mb_{\gG}(V_j) = \{V_k \in V\setminus\{V_j\} : V_k \rdedge V_j, V_k \ldedge
  V_j,~\text{or}~V_k \rdedge V_l \ldedge V_j~\text{for some}~V_l \in
  V\}.
}\]
If $K$ is an ancestral set that contains $V_j$ but none of its
children, it is easy to see that the Markov boundary of $V_j$ in $K$
is precisely its parents (and thus does not depend on $K$):
\[\fitdisplay{
  \mb_{\gG}(V_j, K) = \pa_{\gG}(V_j) = \{V_k \in V: V_k \rdedge V_j\}.
}\]
Therefore, the definition of the ordered local Markov model for DAGs is
consistent with that in \textcite[p.\ 50]{lauritzenGraphicalModels1996}.

\subsection{Marginal independence (MI) property}
\label{sec:conn-mark-models}

The next model only requires the marginal/unconditional independences
in the Markov model. %

\begin{definition}
  The \emph{marginal independence model} with respect to $\gG \in \DMGc(V)$ is
  defined as
  \[\fitdisplay{
    \statmodelMI (\gG, \mathbb{V}) = \{\P \in \statmodel (\mathbb{V}):
    \textnot J \mconnarc K \ingraph{\gG} \Longrightarrow J
    \independent K \underdist{\P}~\text{for all disjoint}~J,K
    \subset V \}.
  }\]
\end{definition}

When the graph $\gG \in \BDG(V)$ is
bidirected, this reduces to the \emph{connected set Markov property}
in \textcite{richardsonMarkovPropertiesAcyclic2003} which says every
connected set (via bidirected edges) is independent of its
non-neighbours.

\subsection{Factorization (F) and exogenous factorization (EF)
  properties}
\label{sec:fact-prop}

\begin{definition}
  For a DAG $\gG\in\DAG(V)$, the \emph{factorization model}
  is defined as
  \[\fitdisplay{\begin{aligned}
    \statmodelF(\gG,\mathbb V)=\Big\{\P\in\mathbb P(\mathbb V):
      \P(dv)=\prod_{j=1}^d
      \P(dv_j\mid v_{\pa_{\gG}(j)})\Big\}.
  \end{aligned}}\]
\end{definition}

Some authors refer to a probability distribution in the above model as
a \emph{Bayesian network}, a terminology due to
\textcite{pearlBayesianNetworksModel1985}.
Next, we give an extension of this definition to exogenously bidirected ADMGs.

\begin{definition}
  Let $E = V_{\mathcal{E}} \subseteq V$ be the set of exogenous vertices for an exogenously
  bidirected ADMG $\gG\in\UnADMGc(V)$. The \emph{exogenous
    factorization model} is defined as
  \[\fitdisplay{\begin{aligned}
    \statmodelEF(\gG,  \mathbb V )=\Big\{\P\in  \mathbb P (  \mathbb V
    ): & \P ( d v )=\P(  dv_{\mathcal E} )\prod  _{V_j\notin E}
           \P( d v _j\mid v_{\pa_{\gG}(j)}) ,\\&\marg_E(\P)\in  \statmodelGM (\marg_E(\gG),  \mathbb V_{\mathcal{E}} ) \Big\},
  \end{aligned}}\]
  where $\P(dv_{\mathcal{E}})=\marg_E(\P)(dv_{\mathcal{E}})$. (Since $E$ is ancestral,
  $\marg_E(\gG)$ is the induced subgraph on $E$.)
\end{definition}

It is easy to see that $\statmodelEF(\gG, \mathbb{V}) =
\statmodelF(\gG, \mathbb{V})$ if $\gG \in \DAG(V)$ and
$\statmodelEF(\gG, \mathbb{V}) = \statmodelGM(\gG, \mathbb{V})$ if
$\gG \in \BDG(V)$. Thus exogenous factorization generalizes DAG
factorization and the Markov property for bidirected graphs.
The requirement that the
marginal distribution of $E$ is global Markov is not essential and can
be replaced by other equivalent definitions (see
\Cref{fig:relation-statmodel-bdg}).

\subsection{Nested Markov (NM) property}
\label{sec:nest-mark-prop}

Throughout this subsection every variable has a nonempty finite or
countable state space.
To describe the nested Markov property, we will need to introduce a
new class of graphs. Let $\ADMGc(V,W)$ collect the set of
\emph{conditional ADMGs}:
\[\fitdisplay{\begin{aligned}
  \ADMGc(V,W) = \{\gG \in \ADMG(V \cup W): &V_j \bdedge V_j
  \ingraph{\gG}~\text{and}~\textnot V \cup W \halfstraigfull
                                             W_k \\
  &\text{for all}~V_j \in V,W_k \in W\}.
\end{aligned}}\]
Here $\ADMG(V \cup W)$ denotes the acyclic subclass of
$\DMG(V \cup W)$, without requiring all bidirected loops.
 Because there are no arrowheads pointing to vertices in $W$, one may
refer to them as ``fixed'' vertices\footnote{So it is perhaps more appropriate to
  call these graphs ``fixed ADMGs''.} and draw them in the graph with
boxes, as done in \textcite{richardsonNestedMarkovProperties2023}. By
definition, $\ADMGc(V,\emptyset) = \ADMGc(V)$.

The nested Markov property is defined through the \emph{fixing}
operator that applies to product spaces, probability distributions/kernels and
conditional ADMGs
\parencite{richardsonNestedMarkovProperties2023}. First, for any $V_j
\in V$, $\fix_{V_j}(\mathbb{V})=\mathbb{V}_{-j}$ because $V_j$ will
be ``fixed''. Next, when acting on a graph $\gG\in\ADMGc(V,W)$, the
operator $\fix_{V_j}:\ADMGc(V,W)\to\ADMGc(V_{-j},W\cup\{V_j\})$
 removes all edges with an arrowhead into $V_j$ and keeps all other
edges.
Not all fixing operations are ``legal''. We say $V_j\in V$ is
\emph{fixable} in $\gG$ if there is no $V_k\in V$ such that
$V_j\rdpath V_k$ and $V_j\samedist V_k\ingraph{\gG}$. In other
words, none of its descendants in $V$ is in the same district as
$V_j$.

Finally, let $\p(v\mid w)$ be a probability mass kernel for
$\gG\in\ADMGc(V,W)$, and suppose $V_j$ is fixable. Denote the
descendants and non-descendants of $V_j$ as
\[
  D=\de_{\gG}(V_j),\qquad N=V\setminus(D\cup\{V_j\}).
\]
The fixing operator on probability distributions/kernels is defined as
\begin{equation}\label{eq:conditional-product-fixing}
  (\fix_{V_j=v_j}(\gG,\p))(n,d\mid w)
  =\p(n\mid w)\p(d\mid v_j,n,w)
\end{equation}
for some chosen version of the kernel $\p(d \mid v_j,n,w)$ to allow
the fixed kernel to be defined at zero-probability conditioning events;
whenever $\p(v_j,n\mid w)>0$, the right-hand side equals $\p(v\mid
w)/\p(v_j\mid n,w)$.
Thus fixing preserves the marginal law of the nondescendants and
uses the conditional law of the descendants at $V_j=v_j$.
Its connection with intervention under the causal Markov model is
established in \Cref{prop:fixing} and \eqref{eq:invariance-fixing} below.
We often omit $\gG$ when clear from context and write
\[
  \fix_{V_j}(\p)
  =(\fix_{V_j=v_j}(\p):v_j\in\mathbb V_j)
\]
for the family of fixed kernels.

For a sequence of distinct vertices $J = V_{\sJ} =
(V_{j_1},\dots,V_{j_n})$, define
\[\fitdisplay{
  \fix_J = \fix_{V_{j_n}}  \circ \dots \circ \fix_{V_{j_1}} ,
}\]
so vertices are fixed in the listed order, updating both the graph and
the kernel at each step (in set operations, $J$ denotes the underlying
set of this sequence). %
We say $J$ is a valid fixing sequence in $\gG$ if $V_{j_m}$ is
fixable in $\fix_{V_{j_{m-1}}}\circ\cdots\circ
\fix_{V_{j_1}}(\gG)$ for all $m = 1,\dots,n$, with the convention that
the empty sequence is valid and $\fix_\emptyset$ is the identity. Not all
permutations of $J$ are valid, but for strictly positive nested Markov
distributions all valid permutations of $J$ define the same fixed kernel
\parencite[Theorem 31]{richardsonNestedMarkovProperties2023}; see also the
remark at the end of \Cref{sec:proof-sketch-crefthm-ne}.

In a nutshell, the nested Markov model requires that the
probability distribution after fixing satisfies an extended global
Markov property with respect to the fixed graph
\parencite[Definitions 4,12,13]{richardsonNestedMarkovProperties2023}.
Consider disjoint subsets $V_{\sK},V_{\sL},V_{\sM}\subseteq V$.
We define
\[\fitdisplay{
 V_{\sK}\independent V_{\sL}\mid V_{\sM}
 \underdist{\fix_{V_{\sJ}}(\P)}
}\]
to hold if either $V_{\sK}\cap V_{\sJ}=\emptyset$ and
\[\fitdisplay{
 \fix_{V_{\sJ}=v_{\sJ}}(\p)
   (v_{\sK}\mid v_{(\sL\cup\sM)\setminus\sJ})
 \text{ is a function only of }v_{\sK}\text{ and }v_{\sM},
}\]
or $V_{\sL}\cap V_{\sJ}=\emptyset$ and the symmetric condition holds
with $\sK$ and $\sL$ exchanged. Thus at least one of $V_{\sK}$ and
$V_{\sL}$ contains no fixed vertices.

\begin{definition} \label{def:nested-markov}
  We say $\P \in \mathbb P(\mathbb{V})$, under the state-space and fixing
  conventions above, is \emph{nested Markov} with
  respect to $\gG \in \ADMGc(V)$ if, for every valid fixing sequence
  $J=V_{\sJ}$, including the empty sequence, and for
  all disjoint $V_{\sK}, V_{\sL}, V_{\sM} \subseteq V$,
  \begin{align}
       \textnot V_{\sK} \mconn V_{\sL} \mid V_{\sM}
       \ingraph{\widetilde{\fix}_{V_{\sJ}}(\gG)}
       \Longrightarrow V_{\sK} \independent V_{\sL} \mid V_{\sM}
       \underdist{\fix_{V_{\sJ}}(\P)}, \label{eq:cadmg-global-markov}
  \end{align}
  where $\widetilde{\fix}_{V_{\sJ}}(\gG)$ is the graph
  $\fix_{V_{\sJ}}(\gG)$ with the additional edges $V_j \bdedge V_k$
  for all $V_j,V_k \in V_{\sJ}$. Let $\statmodelNM(\gG, \mathbb{V})$
  collect all such distributions.
\end{definition}

\subsection{Augmentation (A) criterion}
\label{sec:augm-crit}

The augmentation criterion links statistical models associated with
directed graphs with those associated with undirected graphs. To this
end, let us introduce some additional notation. Let $\UG(V)$ denote
the collection of all simple undirected graphs with vertex set $V$;
specifically, $\UG(V)$ contains all graphs $\gG' = (V, \sE)$ such that
$\sE \subseteq V \times V$, $(V_j, V_j) \not \in \sE$,
and $(V_j,V_k) \in \sE$ implies that $(V_k,V_j) \in \sE$ for all
$V_j,V_k \in V$. This definition is not different from a bidirected
graph besides the requirement of no self-loops, but the semantics of
undirected and bidirected graphs are different in terms of graph
separation. Specifically, for
$\gG' \in \UG(V)$ and disjoint subsets $J, K, L \subset V$, we say
$L$ \emph{separates} $J$ and $K$ in $\gG'$ and write
\[\fitdisplay{
  \textnot J \uconn K \mid L \ingraph{\gG'},
}\]
if every path from a vertex in $J$ to a vertex in $K$ contains a
non-endpoint in $L$. The Markov model associated with an
undirected graph $\gG' \in \UG(V)$ is defined as
\[\fitdisplay{\begin{aligned}
  \statmodelGM (\gG', \mathbb{V}) = \{\P \in \mathbb P (\mathbb{V}):
  &\textnot J \uconn K \mid L \ingraph{\gG'} \Longrightarrow J
  \independent K \mid L \underdist{\P}\\
  &\text{for all disjoint}~J,K,L \subset V\}.
\end{aligned}}\]

Consider the following \emph{augmentation} map from directed mixed
graphs to undirected graphs:
\[\fitdisplay{\begin{aligned}
  \augg: \DMGc(V) &\to \UG(V), \\
  \gG &\mapsto \gG',
\end{aligned}}\]
where $\gG' = \augg(\gG)$ is an undirected graph with the same vertex
set $V$ such that
\begin{equation*}
  V_j \udedge V_k \ingraph{\gG'} \Longleftrightarrow V_j
  \colliderconn V_k \ingraph{\gG}~\text{for all}~V_j, V_k \in V, V_j
  \neq V_k.
\end{equation*}
That is, $V_j$ is connected to all vertices in its Markov boundary.
When this map is restricted to DAGs,
this is often known as \emph{moralization} in the literature
because it connects any two parents with the same child
\parencite{lauritzenGraphicalModelsAssociations1989,frydenbergChainGraphMarkov1990}. For
ADMGs, the augmentation criterion below is introduced in
\textcite{richardsonMarkovPropertiesAcyclic2003}.

\begin{definition}
  The augmentation model for $\gG \in \ADMGc(V)$ is defined as
  \[\fitdisplay{\begin{aligned}
    \statmodelA(\gG, \mathbb{V}) = \{\P \in \mathbb P (\mathbb{V}):&
    \marg_J(\P) \in \statmodelGM (\augg\circ \marg_J(\gG),
    \marg_J(\mathbb{V}))\\&\text{for all ancestral}~J \subseteq V\}.
  \end{aligned}}\]
\end{definition}

\subsection{Pairwise (PE), clique (CE), and noise (NE) expansions}
\label{sec:pairw-cliq-noise}

One way to define statistical models associated with a general ADMG is
through expanding the graph to ``simpler graphs''. First, let us
define graph expansion, which is simply the pre-image of graph
marginalization. Specifically, given $\gG \in \ADMGc(V)$, define
 \[\fitdisplay{
  \expand(\gG) = \{\gG' \in \ADMGc (V'): V' \supseteq V,\ |V'|<\infty,\  \marg_V(\gG') =
  \gG\}.
}\]
Obviously, graph marginalization is not injective, so $\expand(\gG)$
is an infinite set of graphs that can marginalize to $\gG$.

There are several possible ``natural'' definitions that pick a
specific element of $\expand(\gG)$ as ``the'' expanded graph. Consider
$V = \{V_1,\dots,V_{d}\}$ and $\gG \in \DMGc(V)$. The \emph{pairwise
  expansion} replaces every bidirected edge by a latent common
parent. Formally, the pairwise expansion graph $\expandP(\gG)$ has
vertex set $V \cup E$ with $E = \{E_{jk}: V_j \bdedge V_k
\ingraph{\gG}, j < k\}$ and the following edges:
\[\fitdisplay{\begin{aligned}
  E_{jk} \rdedge V_j, E_{jk} \rdedge V_k
  \ingraph{\expandP(\gG)},&~\text{for all}~1 \leq j < k \leq
                            d~\text{such that}~V_j \bdedge V_k
                            \ingraph{\gG}, \\
  V_j \rdedge V_k \ingraph{\expandP(\gG)},&~\text{for all}~j, k \in [d]
                                            ~\text{such that}~V_j \rdedge V_k
                                            \ingraph{\gG}.
\end{aligned}}\]
The \emph{clique expansion}
replaces every bidirected clique (in which every two vertices are connected by a
bidirected edge) by a latent common parent. Formally, if we let
$\mathcal{C}(\gG)$ denote (the vertex indices of) all bidirected
cliques in $\gG$, that is,\footnote{One can instead use only inclusion-maximal bidirected
  cliques, without changing the clique expansion model. The
  definition employed here simplifies our proof in \Cref{sec:technical-proofs} that
  the clique expansion model is complete.%
}
\[\fitdisplay{
  \mathcal{C}(\gG) = \{\sJ \subseteq [d] : V_j \bdedge V_k~\text{for all}~j,k
  \in \sJ\},
}\]
then the clique expansion graph $\expandC(\gG)$ has vertex set $V\cup
E$ with $E = \{E_{\sJ}: \sJ \in \mathcal{C}(\gG)\}$ and the following edges:
\[\fitdisplay{\begin{aligned}
  E_{\sJ} \rdedge V_j \ingraph{\expandC(\gG)},&~\text{for all}~j \in
                                                \sJ \in \mathcal{C}(\gG), \\
  V_j \rdedge V_k \ingraph{\expandC(\gG)},&~\text{for all}~j, k \in [d]
                                            ~\text{such that}~V_j \rdedge V_k
                                            \ingraph{\gG}.
\end{aligned}}\]
For $\gG\in\ADMGc(V)$, the pairwise and clique expansion graphs are DAGs.

The \emph{noise expansion}, on the other hand, results in an exogenously bidirected
 graph where the bidirected and directed edges are
``separated''. Formally, the noise expansion graph $\expandN(\gG)$ has
vertex set $V \cup E$ with $E = \{E_1, \dots, E_d\}$ and the following edges:
\[\fitdisplay{\begin{aligned}
  E_j \rdedge V_j \ingraph{\expandN(\gG)},&~\text{for all}~j \in [d],
  \\
  E_j \bdedge E_k \ingraph{\expandN(\gG)},&~\text{for all}~j, k \in
                                            [d]~\text{such that}~V_j
                                            \bdedge V_k \ingraph{\gG}, \\
  V_j \rdedge V_k \ingraph{\expandN(\gG)},&~\text{for all}~j, k \in [d]
                                            ~\text{such that}~V_j \rdedge V_k
                                            \ingraph{\gG}.
\end{aligned}}\]

\begin{definition}
  For $\gG\in\ADMGc(V)$, the \emph{pairwise expansion model},
  \emph{clique expansion model}, and \emph{noise expansion model} are
  defined as the $V$-marginals of the Markov models for the
  corresponding expanded graphs:
  \[\fitdisplay{\begin{aligned}
    \statmodelPE(\gG,\mathbb V )
      &=\marg_V\statmodelGM
        ( \expandP(\gG),\mathbb V\times[0,1]^{b(\gG)}) , \\
    \statmodelCE(\gG,\mathbb V )
      &=\marg_V\statmodelGM
        ( \expandC(\gG),\mathbb V\times[0,1]^{|\mathcal C(\gG)|}) , \\
    \statmodelNE(\gG,\mathbb V )
      &=\marg_V\statmodelGM
        ( \expandN(\gG),\mathbb V\times[0,1] ^{|V|}),
   \end{aligned}}\]
  where $b(\gG)=|\{(j,k):V_j\bdedge V_k
  \ingraph{\gG},\ j<k\}|$. The latent coordinates have Borel state
  space $[0,1]$ and arbitrary laws.
 \end{definition}

This definition assumes that the latent variables are all supported on
the unit interval, which suffices because standard Borel latent spaces
admit measurable encodings into $[0,1]$.

\subsection{Nonparametric equation (E) systems}
\label{sec:nonp-syst}

\begin{definition}[Nonparametric system] \label{def:np-system}
  Consider $\gG\in\ADMGc(V)$. The \emph{nonparametric
  equation system} $\statmodelE(\gG,\mathbb V)$ consists
  of laws of the recursive solution of
   \begin{equation}
    \label{eq:np-system}
    V_j = f_j(V_{\pa_{\gG}(j)}, E_j),~j=1,\dots,d
  \end{equation}
  for some (measurable) functions $f_j: \mathbb{V}_{\pa_{\gG}(j)}
  \times [0,1] \to
  \mathbb{V}_j$, $j=1,\dots,d$ and a random vector
  $E = (E_1,\dots,E_{d}) \in [0,1]^d$ whose joint distribution $\Q$ is
  arbitrary subject to satisfying the marginal independence property with respect to the bidirected component of
  $\gG$, that is, for all disjoint $\sJ, \sK \subset [d]$, we have
  \begin{equation}
    \label{eq:independent-error}
    V_{\sJ} \nobdedge V_{\sK} \ingraph{\gG}
    \Longrightarrow E_{\sJ} \independent E_{\sK} \underdist{\Q}.
  \end{equation}
 \end{definition}

This definition is closely related to the ``semi-Markovian'' causal model
in \textcite[p.\ 30]{pearlCausality2009}, but there are some subtle
distinctions. First, Pearl does not explicitly state
\eqref{eq:independent-error} as the Markov condition on the
distribution of the noise variables and just calls the model
semi-Markovian if the noises are correlated. In another definition of
semi-Markovian models, \textcite[pp.\
542--543]{bareinboimPearlHierarchyFoundations2022}
define its causal diagram by adding a bidirected edge between $V_j$
and $V_k$ if the corresponding noise variables are
correlated. However, equation \eqref{eq:independent-error} is
stronger than pairwise independence for missing bidirected edges: it
requires independence of the corresponding subvectors for any two
disjoint vertex sets with no bidirected edge between them. Second, Pearl intends to
interpret \eqref{eq:np-system} not just as a statistical model but
also as a causal model. A formal treatment of causal models
usually requires the potential outcomes of $V$ under
interventions. This is investigated in \Cref{sec:causal-model} below.

\section{Causal Markov model and the nested Markov property}
\label{sec:causal-model}

The nonparametric equation system gives a natural definition of
potential outcomes using recursive substitution
\parencite{pearlCausality2009,richardson2013single}. In this Section,
we will introduce this causal model and use it to prove that the
nonparametric equation system is nested Markov as formally stated
in the next Theorem. We will restrict our discussion to finite or
countable state spaces; latent variables and
structural noises need not be discrete. This avoids technical measure
theory details which can obfuscate the correspondence of key graphical
and statistical/causal concepts. We expect analogous results for
Polish state spaces under suitable consistency and regularity
assumptions, including continuous counterfactual sample paths, but do
not pursue this extension here.

\begin{theorem} \label{thm:admg-ne-nm}
  For $\gG \in \ADMGc(V)$ and any finite or countable product space
  $\mathbb{V}$, we have
  $$\statmodelE(\gG,\mathbb V)\subseteq\statmodelNM(\gG,\mathbb V).$$
   In other words, the implication $\text{E} \Rightarrow \text{NM}$ in
  \Cref{fig:relation-statmodel-admg} holds.
\end{theorem}

\subsection{Causal model}
\label{sec:causal-model-1}

Let us first define what we mean by a causal model. We have used
``statistical model'' to refer to a collection of probability
distributions on a set of random variables. Likewise, a causal model
is a collection of probability distributions on all ``potential
outcomes''. Specifically, let the random variable $V_j(v_{\sI})$
denote the \emph{potential outcome} of $V_j$ under an intervention
that sets $V_{\sI}$ to $v_{\sI}$, $j \in [d]$, $\sI \subseteq
[d]$. We use the convention in \textcite{richardson2013single}: when $j \in \sI$,
$V_j(v_{\sI})$ is not set equal to $v_j$; the imposed values are
separate fixed quantities. The \emph{potential outcome schedule}
$V(\cdot)$ is the collection of all potential outcomes:
\[\fitdisplay{
  V(\cdot) = (V_j(v_{\sI}): j \in [d], \sI \subseteq [d], v_{\sI} \in
  \mathbb{V}_{\sI}).
}\]
The schedule takes values in
\[\fitdisplay{
  \mathbb{V}(\cdot)=\prod_{\sI\subseteq[d]}
    \ \prod_{v_{\sI}\in\mathbb{V}_{\sI}}\mathbb{V},
  \qquad \mathbb{V}_{\emptyset}=\{\emptyset\}.
}\]
The observational vector is the empty-intervention coordinate
$V = V(\emptyset)$. Let $\statmodel(\mathbb{V}(\cdot))$ denote the
class of probability distributions on the potential outcome schedule,
equipped with its product $\sigma$-algebra.

\begin{definition}[Causal model]
  We say $\P \in \statmodel(\mathbb{V}(\cdot))$ is
  \emph{causal} if the following \emph{consistency} property holds: for
  all disjoint $\sI, \sI' \subseteq [d]$ and fixed $v \in \mathbb{V}$
  \begin{equation}
    \label{eq:consistency}
    \P( V(v_{\sI}, v_{\sI'}) = V(v_{\sI})
  \mid V_{\sI'}(v_{\sI}) =
    v_{\sI'}) = 1,
   \end{equation}
  whenever $\P(V_{\sI'}(v_{\sI}) = v_{\sI'}) > 0$. Let $\causalmodel(\mathbb{V})$
  denote all such probability distributions. We say a subset of
  $\causalmodel(\mathbb{V})$ is a \emph{causal model}.
\end{definition}

Thus, a causal model is a statistical model on the potential outcome schedule that satisfies the consistency property
\eqref{eq:consistency}. This property is not new and can be found in,
for example, \textcite{malingsky2019potential}. It generalizes the usual notion of
consistency or stable unit treatment value \parencite{rubin80comment} in
causal inference, which says the observed outcome is the same as the
potential outcome under an intervention that ``sets'' a treatment to its
observed value. Note that this definition of a causal model does not
depend on any graphical representation.

\subsection{Causal Markov model}
\label{sec:caus-model-assoc}

Consider an ADMG $\gG \in \ADMGc(V)$ and a nonparametric equation
system as given in \Cref{def:np-system}. Roughly speaking, we can
interpret the equations in
\eqref{eq:np-system} causally by requiring that those equations
still hold in an intervention that sets some of the variables to a
fixed value. %

\begin{definition}[Causal Markov model] \label{def:causal-markov}
 We say a distribution $\P \in \statmodel(\mathbb{V}(\cdot))$ is
 \emph{causal Markov} with respect to $\gG \in \ADMGc(V)$ if the
 following are true:
 \begin{enumerate}
  \item The potential outcomes are consistent with respect to $\gG$ in
    the sense that the next event has $\P$-probability $1$:
\begin{equation}
  \label{eq:basic-consistency}
  V_j(v_{\sI}) = V_j(v_{\pa_{\gG}(j) \cap \sI}, V_{\pa_{\gG}(j) \setminus
    \sI}(v_{\sI})),~\text{for all}~j \in [d], \sI \subseteq [d], v \in
  \mathbb{V}.
\end{equation}
 \item The distribution of the basic potential
outcomes satisfies the marginal independence property with respect to the bidirected component of
$\gG$, that is, for all disjoint $\sJ, \sK \subset [d]$, we
have
\begin{equation}
    \label{eq:independent-basic-po}
     V_{\sJ} \nobdedge V_{\sK} \ingraph{\gG}
    \Longrightarrow V_{\sJ}(v) \independent V_{\sK}(v)
    \underdist{\P}~\text{for all}~v \in \mathbb{V}.\footnote{Note that
  \eqref{eq:independent-error} implies more than
\eqref{eq:independent-basic-po}: under the same missing-bidirected-edge
condition, it also implies $V_{\sJ}(v) \independent V_{\sK}(v')$
for different $v,v'\in\mathbb{V}$. We choose not to include these ``cross-world'' independences
because, as argued by \mbox{%
\textcite{richardson2013single}}\hskip0pt%
, they cannot
possibly be verified by any experiment.}
   \end{equation}
 \end{enumerate}
  The \emph{causal Markov model} associated with $\gG$ is then
  defined as
  \[\fitdisplay{
    \causalmodel(\gG, \mathbb{V}) = \{\P \in \statmodel(\mathbb{V}(\cdot)):
    \text{$\P$ is causal Markov with respect to $\gG$} \}.
  }\]
\end{definition}

We will see in \Cref{thm:equivalent-axioms}
below that this model implies consistency of the potential outcomes.
This definition generalizes the single-world causal model introduced by
\textcite{richardson2013single} in two ways: first, the causal diagram
can be an ADMG instead of just a DAG; second, the primitive objects in
this definition are potential outcomes instead of structural
equations.\footnote{This is already hinted at in \textcite[Definition
  1]{richardson2013single}.}

Note that the directed and bidirected edges play different roles in
this definition. The directed edges represent
direct causal effects, and the bidirected edges represent
exogenous correlation. Importantly, this model does not assume that the
exogenous correlations arise from latent common causes. In the
author's opinion, this is more transparent than the approach taken in
\textcite[Section 4.3]{richardsonNestedMarkovProperties2023} and
implicitly taken in Pearl's work that assumes a causal model with
respect to some unspecified DAG expansion of the ADMG. It is difficult
to conceptualize potential outcomes of the latent variables without
knowing what they are. In contrast, the causal Markov model above only
requires potential outcomes of the variables in the ADMG. See
\Cref{sec:discussion} for further discussion.

The equations in the E model (see \Cref{def:np-system}) give a natural
definition of potential outcomes via the following recursion:
\begin{equation}
  \label{eq:npsem}
  V_j(v_{\sI}) = f_j(v_{\pa(j) \cap \sI}, V_{\pa(j) \setminus
    \sI}(v_{\sI}), E_j),~j=1,\dots,d.
\end{equation}
The distribution of the potential outcome schedule is then entirely
determined by the functions $f_1,\dots,f_d$ and the distribution of
the noise variables $E_1,\dots,E_d$. This is often referred to as the
\emph{structural equation model} \parencite{pearlCausality2009} or
\emph{structural causal model}
\parencite{petersElementsCausalInference2017,bareinboimPearlHierarchyFoundations2022},
although the assumption on the distribution of $E$ is not always
clearly stated; see the remarks in \Cref{sec:nonp-syst}. The
distribution of potential outcomes defined via
\eqref{eq:npsem} is causal Markov with respect to $\gG$:
\eqref{eq:basic-consistency} immediately follows from
\eqref{eq:npsem}, and
\eqref{eq:independent-basic-po} immediately follows from
\eqref{eq:independent-error}. We summarize this observation as a Lemma.\footnote{Note that
  our definition of a causal Markov model is a collection of
probability distributions on the potential outcomes schedule and does
not require defining potential outcomes via structural equations. It is
natural to ask if the reverse of \Cref{lem:npsem-causal} is true. It
is observed in \textcite[p.\
    22]{richardson2013single} that one can use the potential outcomes
    to define structural equations as
  \[\fitdisplay{
    f_j(v_{\pa(j)}, E_j) = V_j(v_{\pa(j)}),~j=1,\dots,d,
  }\]
  where $E_j = (V_j(v_{\pa(j)}): v_{\pa(j)} \in \mathbb{V}_{\pa(j)})$
  collects all basic potential outcomes for $V_j$. If $\mathbb{V}$ is
  finite, the range
  of $E_j$ is a countable product of discrete spaces, so it can be
  measurably encoded in $[0,1]$. Nevertheless, independence of the ``noise'' in
  \eqref{eq:independent-error} does not directly follow from
  single-world independence of the potential outcomes in
  \eqref{eq:independent-basic-po}.}

\begin{lemma} \label{lem:npsem-causal}
  For any $\gG \in \ADMGc(V)$ and product space $\mathbb{V}$, we have
  \[\fitdisplay{
    \statmodelE(\gG, \mathbb{V}) \subseteq \marg_V(\causalmodel(\gG,
    \mathbb{V})).
  }\]
\end{lemma}

\subsection{Properties of the causal Markov model}
\label{sec:caus-ident}

We collect three properties in a potential-outcome calculus and show
that they characterize the causal Markov model. To describe the
single-world Markov property, let us generalize the definition
of single world intervention graphs (SWIGs) in
\textcite{richardson2013single} from
DAGs to ADMGs. Given $\gG \in \ADMGc(V)$, let $\gG(v_{\sI})$ denote
the graph obtained by removing all outgoing edges from $V_{\sI}$
(i.e.\ edges like $V_{\sI} \rdedge \ast$) and relabeling $V_j$ as
$V_j(v_{\sI})$ for all $V_j \in V$.\footnote{We do not consider the
  ``fixed vertex'' $v_i$ for $i \in \sI$ as in
  \textcite{richardson2013single}, because we
  are only interested in the distribution of $V(v_{\sI})$ here.} Let
$V_{-j}$ denote the complement of $V_j$ in $V$ and $V_{-\sJ}$ denote
the complement of $V_{\sJ}$.

\begin{definition}[Potential-outcome calculus]
\label{def:po-calculus}
  A law $\P \in \mathbb{P}(\mathbb{V}(\cdot))$ satisfies the
  \emph{potential-outcome calculus} for $\gG\in\ADMGc(V)$ if:
  \begin{enumerate}
  \item \emph{Consistency:} $\P\in\causalmodel(\mathbb V)$, that is,
    \eqref{eq:consistency} holds.
  \item \emph{Causal irrelevance:} for disjoint $\sK,\sL\subseteq[d]$
    and any $\sJ\subseteq[d]$ (which may intersect $\sK$ or $\sL$),
    \[
      \textnot V_{\sL}\rdpath V_{\sJ}\mid V_{\sK}\ingraph{\gG}
      \Rightarrow
      \P(V_{\sJ}(v_{\sK},v_{\sL})=V_{\sJ}(v_{\sK}))=1,
      \quad\text{for all}~v_{\sK}\in\mathbb V_{\sK},v_{\sL}\in\mathbb V_{\sL}.
    \]
  \item \emph{Single-world Markov property:} for every
    $\sI\subseteq[d]$ and $v_{\sI}\in\mathbb V_{\sI}$,
    \[\fitdisplay{\marg_{V(v_{\sI})}(\P)\in
    \statmodelGM(\gG(v_{\sI}),\mathbb V).}\]\end{enumerate}
\end{definition}

\begin{theorem} \label{thm:equivalent-axioms}
  A law $\P$ belongs to $\causalmodel(\gG,\mathbb V)$ if and only if
  it satisfies the potential-outcome calculus in \mbox{%
\Cref{def:po-calculus}}\hskip0pt%
.
\end{theorem}

Consistency justifies calling $\causalmodel(\gG,\mathbb V)$ a causal
model; causal irrelevance says that no directed paths means no causal
effect. The single-world Markov property generalizes similar results
for DAGs, for example, Theorem 1.4.1 in \textcite{pearlCausality2009}
and Proposition 11 in
\textcite{richardson2013single}. \Cref{thm:equivalent-axioms} shows that
the potential-outcome calculus---which is equivalent to Pearl's
do-calculus \parencite{malingsky2019potential}---is complete in the sense
that it is equivalent to the causal model in \Cref{def:causal-markov}.

The next property establishes the connection between fixability and
intervention. It generalizes Proposition 5 in
\textcite{shpitserMultivariateCounterfactualSystems2022a} (where ADMGs are
interpreted as DAGs with latent variables) to the causal ADMG model in
\Cref{def:causal-markov}.

\begin{proposition}[Fixing and intervention]\label{prop:fixing}
Suppose $\P\in\causalmodel(\gG,\mathbb V)$ and write $\P(V_{\sJ} =
v_{\sJ} \mid V_{\sK} = v_{\sK})$ as $\P(v_{\sJ} \mid v_{\sK})$. If $V_j$ is fixable in
$\gG$, then
\begin{equation}\label{eq:fixing-intervention}
  \P\bigl(V_j(v_j)=\tilde v_j,\,
          V_{-j}(v_j)=v_{-j}\bigr)
  =
  \P(v_j,v_{-j}) \frac{\P(\tilde{v}_j\mid v_{\mathcal{B}})}
       {\P(v_j\mid v_{\mathcal{B}})}
\end{equation}
for $\mathcal{B}=\mbg_{\gG}(j)$ and all $v\in\mathbb V$ and $\tilde v_j\in\mathbb V_j$ such that
$\P(v_j\mid v_{\mathcal{B}})>0$ (the quotient is interpreted as zero when
$\P(v_{\mathcal B})=0$). %
\end{proposition}

The proofs of \Cref{thm:equivalent-axioms} and \Cref{prop:fixing} can
be found in \Cref{sec:technical-proofs}.

\subsection{\texorpdfstring{Proof sketch of \Cref{thm:admg-ne-nm}}{Proof sketch of Theorem 3}}
\label{sec:proof-sketch-crefthm-ne}

We conclude this Section with a proof sketch for
\Cref{thm:admg-ne-nm} by ``lifting'' the statistical model to
the causal model. Consider any
$\P_V\in\statmodelE(\gG,\mathbb V)$, and construct its potential
outcomes using the structural equations in \eqref{eq:npsem}.
By \Cref{lem:npsem-causal}, their joint law
$\P\in\causalmodel(\gG,\mathbb V)$ satisfies
$\marg_V(\P)=\P_V$.

As a corollary of \Cref{prop:fixing}, if $V_j$ is fixable and
$\P(v_j\mid V_{\mathcal B})>0$
almost surely, summing
\eqref{eq:fixing-intervention} over $\tilde v_j$ gives
\begin{equation}\label{eq:identify-by-fixing}
  \P(V_{-j}(v_j)=v_{-j})
  =
  \frac{\P(v_j,v_{-j})}{\P(v_j\mid v_{\mathcal B})}
  =
  \bigl(\fix_{V_j=v_j}(\P_V)\bigr)(v_{-j}).
\end{equation}
Thus fixing identifies the distribution of $V_{-j}(v_j)$. %
By repeating this argument along a valid fixing sequence $J=V_{\sJ}$, we obtain
\begin{equation}\label{eq:invariance-fixing}
  \fix_{V_{\sJ}=v_{\sJ}}(\P_V)
  =
  \marg_{V_{-\sJ}(v_{\sJ})}(\P),
  \qquad v_{\sJ}\in\mathbb V_{\sJ}.
\end{equation}
It remains to establish the extended conditional independences
required by the nested Markov property. This step uses
consistency, causal irrelevance, and the single-world Markov
property in \Cref{thm:equivalent-axioms}.
The details are given in \Cref{sec:technical-proofs}.

As a final remark, note that the equality in
\eqref{eq:invariance-fixing} immediately implies that different orders
of fixing, as long as each fixing sequence is valid, identify the same
interventional distribution. In other words, fixing is
order-independent for $\P \in \statmodelE(\gG,\mathbb
V)$. \textcite[Theorem 31]{richardsonNestedMarkovProperties2023} have
further proved this order-independence for $\P \in
\statmodelNM(\gG,\mathbb V)$. Intuitively, this is true
because $\statmodelNM$ contains the same equality constraints as
$\statmodelCE$ \parencite{evansMarginsDiscreteBayesian2018} and thus
$\statmodelE$ (as $\statmodelCE \subseteq \statmodelE \subseteq \statmodelNM$
per \Cref{thm:relation-statmodel}).

\section{Discussion}
\label{sec:discussion}

\subsection{ADMGs and causal inference}

We have mainly considered different statistical models (collections
of probability distributions of $V$) associated with ADMGs, but many such
models are closely related to causal models (collections of probability
distributions of $V$ and all its potential outcomes) as shown in
\Cref{sec:causal-model}. In fact, \Cref{sec:causal-model} proves the
highly non-trivial \Cref{thm:admg-ne-nm} about nonparametric equation
models by ``lifting'' them to causal models. This allows us to break down the proof of
$\text{NE} \Rightarrow \text{NM}$ into several intuitive steps.
Intuitively, this strategy is possible because a nonparametric equation
model has at least one causal explanation.

Of course, it is not new to use ADMGs for causal inference. In fact, the first paper to introduce causal diagrams and
path analysis \parencite{wrightRelativeImportanceHeredity1920}
already used ADMGs, rather than DAGs, over a century ago:
two types of edges were needed to describe causal dependence
and statistical correlation in a linear structural equation
model. This tradition has been kept in
social science; see
e.g.\ \textcite{bollenStructuralEquationsLatent1989} and the popular
LISREL software and its modern implementations
\parencite{Rosseel2012,joreskogLISREL10Windows2018}. Moreover,
ADMGs are used in the
groundbreaking do-calculus \parencite{pearlCausalDiagramsEmpirical1995,pearlCausality2009} and the
ID algorithm for causal identification
\parencite{tianTestableImplicationsCausal2002a, huangValtorta2006complete,
shpitserPearl2006joint, richardsonNestedMarkovProperties2023}.

But here we would like to make a different philosophical point: causal
inference can and should be \emph{entirely based on ADMGs}. More
specifically, we intend to criticize the following ``latent DAG
interpretation'' of ADMGs that is commonly found in written and verbal
communications about causal graphs:
\begin{quote}
An ADMG is just a convenient shortcut to represent some unspecified large
causal DAG that generates the data.
\end{quote}
Putting this differently, we argue that the ADMG-based theory of
causality is \emph{a proper generalization rather than a derivative} of the
DAG-based theory.

At face value, the ``latent DAG interpretation'' imposes an additional
latent causal representability assumption: once an ADMG and its
vertices are chosen, it requires the existence of a compatible latent
DAG generating the residual associations.
\Cref{thm:statmodelE-agnostic} reveals the fundamental statistical difference between the ``latent DAG
interpretation'' (corresponding to the CE model) and our preferred
interpretation via noise expansion (the NE/E model): the CE model uses
Markov models on DAGs as the base model, whereas the NE/E model
uses Markov models on exogenously bidirected ADMGs. So they
correspond to quite different philosophical stances: the ``latent DAG
interpretation'' is essentially statistical
reductionism---every variable is a result of some earlier or
lower-level features and some statistical noise, while our ADMG causal
Markov model focuses on causal relationships between the variables in the graph without attempting to explain how exogenous
correlations arise. Mathematically, a benefit of the ADMG-based
generalization is that it shows the familiar potential outcome
calculus is a complete characterization of the axiomatic definition of
ADMG-based causal Markov models.

In other words, our causal model is non-reductionist in the sense that
it does not commit to Reichenbach's Common Cause Principle. The Common
Cause Principle says that if two variables are correlated and neither is a cause of the
other, then they must have a common cause that renders them
conditionally independent. In consequence, users of our causal model can focus on the variables under investigation and on justifying the
\emph{missing} directed and bidirected edges, rather than explaining why
any bidirected edge is present, because causal identification (through
the do-calculus or ID algorithm) relies on the missing edges. By using ADMGs and
drawing bidirected edges directly, practitioners are encouraged to
deliberate over missing bidirected edges, as in the iterative
confounder-selection procedure of
\textcite{guoConfounderSelectionIterative2023}.

Despite what has been said, the E/NE and CE models both generally
impose inequality as well as equality constraints (nested Markov constraints) on ADMGs, and it is
not hard to show that they coincide when every bidirected district is a clique. What we are really arguing is that
it is unhelpful to think about ``\emph{the} causal DAG'' that
generates the data. Instead, it is more modest and productive to think
about a nested sequence of ADMGs with more and more variables that
explain the data, acknowledging that there may always remain some
confounding relationships (represented by bidirected edges) whose
exact mechanism is unknown and not needed for the question at hand.

Advocating NE/E as the default statistical/causal model for ADMGs does
not mean we discourage investigators from considering and modeling latent
variables; in fact, it is quite the opposite. The main point here is
that ``latent variable'' is not a property of a statistical or causal
model itself, but a property of a model \emph{coupled with} a
dataset. This point is most visible in data integration problems where
several datasets with different measurements are available. Thus,
adopting the NE/E interpretation---a marginal perspective of graphical
models---encourages researchers to focus on which variables, measured
in a particular dataset or not, should be modeled. When there are
scientific reasons to believe two variables share a common cause,
practitioners are encouraged to include that variable explicitly in
the graph even if it is unmeasured. Factor models, latent mixture
models, and proximal identification strategies remain valuable when
their assumptions are justified \parencite[see
e.g.][]{tchetgentchetgenIntroductionProximalCausal2024}.

\subsection{Future research}
\label{sec:open-problems}

There are some important open problems to consider in future
work. First, it would be interesting to understand the inequality
constraints implied by the E/NE model, in addition to the equality
constraints in the nested Markov model. Second, ADMGs
can also be used to describe quantum mechanics models, with connections
to nested Markov models
 \parencite{navascuesInflationTechniqueCompletely2020}. A quick
investigation shows that the E/NE model neither contains nor is
contained by the quantum mechanics model: the E/NE model has a more
relaxed interpretation of bidirected graphs but a local interpretation
of directed edges. It would be interesting to study their relations
further and consider super-models that contain both of them. Third,
many modern causal inference methods use graphical diagrams to identify
the causal estimands of interest and then estimate those parameters
using influence-function based methods. These methods typically
require pathwise differentiability of the estimands within the model,
and it would be interesting to study that for the E/NE model defined
here.

%% file: appendix.tex
\section{Technical proofs}
\label{sec:technical-proofs}

Unless otherwise specified, all product spaces
$\mathbb{V} = \mathbb{V}_1 \times \dots \times \mathbb{V}_d$
have nonempty standard Borel coordinate spaces, as introduced in
\Cref{sec:marginalization}. Products of probability kernels are
composed in a topological order of the graph.

We use two standard facts about these spaces. First, every standard
Borel space $S$ is Borel isomorphic to a Borel subset of $[0,1]$
\parencite[Theorem~1.8]{kallenberg21_found_moder_probab}.
We call such a map a \emph{Borel encoding}. In particular, a random
element $V_j$ can be encoded as $E_j=e_j(V_j)$ and recovered as
$V_j=e_j^{-1}(E_j)$.
Second, every probability kernel $K$ from a measurable space $S$
into a nonempty standard Borel space $T$ admits a jointly
measurable map $f:S\times[0,1]\to T$ such that $
  K(B \mid x)=\int_0^1 \mathbf{1}\{f(x,u)\in B\}\,du$
for every $x\in S$ and every measurable $B\subseteq T$
\parencite[Lemma~4.22]{kallenberg21_found_moder_probab}. For real-valued
outcomes, $f$ can be taken to be a conditional quantile function.

\subsection{\texorpdfstring{Proof of \Cref{thm:relation-statmodel}}{Proof of Theorem 1}}
\label{sec:proof-crefthm:r-stat}

As mentioned previously, many implications and equivalences in
\Cref{fig:relation-statmodel} of the main paper are already proved in
the literature. We
will identify the new claims and then prove them in a sequence of
Lemmas. As in the main theorem, comparisons involving NM assume
finite or countable state spaces; the other models retain their
general state spaces. %

Relations in \Cref{fig:relation-statmodel-admg}: it follows from the definition that $\text{PE}
\Rightarrow \text{CE}$, $\text{NM} \Rightarrow \text{M}\Rightarrow \text{MI}$, and $\text{E} \Rightarrow
\text{NE}$ (the last implication requires $\text{EF}=\text{M}$ for
exogenously bidirected graphs; see \Cref{lem:unconfounded-ef-gm}
below). It is shown in \textcite[Theorem 2]{richardsonMarkovPropertiesAcyclic2003} that $\text{LM}
\Leftrightarrow \text{M} \Leftrightarrow \text{A}$ and essentially in
\textcite[Theorem 46]{richardsonNestedMarkovProperties2023} that
$\text{CE} \Rightarrow \text{NM}$. In \Cref{lem:admg-ce-ne} below, it
is shown that $\text{CE} \Rightarrow \text{NE}$. It
follows from
\Cref{lem:admg-ne-e} below that $\text{NE} \Rightarrow \text{E}$ and from
\Cref{thm:admg-ne-nm} in the main text that $\text{E} \Rightarrow
\text{NM}$.

Relations in \Cref{fig:relation-statmodel-unconfounded}: it follows
from
\Cref{lem:unconfounded-e-ef,lem:unconfounded-ef-gm%
} below
that $\text{E} \Leftrightarrow \text{EF} \Leftrightarrow \text{M}$.
The rest of the
relations follow from \Cref{fig:relation-statmodel-admg}.

Relations in \Cref{fig:relation-statmodel-dag}: it is shown in
\textcite[Theorem 3.27]{lauritzenGraphicalModels1996} that $\text{M}
\Leftrightarrow \text{F}$ (although there is a gap in the proof of
\textcite[Proposition 3.25]{lauritzenGraphicalModels1996}; see the
remark after \textcite[Corollary
2]{richardsonMarkovPropertiesAcyclic2003}). By definition,
$\expandP(\gG) = \gG$ because a DAG has no bidirected edges (recall
that we do not consider bidirected loops). So, by definition, $\text{PE}
\Leftrightarrow \text{M}$. The rest of the equivalences and
implications follow from \Cref{fig:relation-statmodel-unconfounded}
(because DAGs are exogenously bidirected).

Relations in \Cref{fig:relation-statmodel-bdg}: it is shown in
\textcite[Theorem 3]{richardsonMarkovPropertiesAcyclic2003} that
$\text{M} \Leftrightarrow \text{MI}$. The rest of the equivalences and
implications follow from \Cref{fig:relation-statmodel-unconfounded}
(because bidirected graphs are exogenously bidirected).

It remains to show that the relations in \Cref{fig:relation-statmodel}
are ``tight'' in the sense that when the two models are not connected
by $\Leftrightarrow$ in \Cref{fig:relation-statmodel}, there exists
some graph in the corresponding class such that the models are not
equal. It suffices to consider the following cases:
\begin{enumerate}
\item When $\gG$ is a DAG, $\text{MI} \Rightarrow \text{M}$ is
  not always true. For example, consider the graph $A \ldedge
  B \rdedge C$, for which the M model contains the additional
  conditional independence $A \independent C \mid B$.
\item When $\gG$ is bidirected, $\text{M} \Rightarrow \text{CE}$ and
  $\text{CE} \Rightarrow \text{PE}$ are not always true. This is
  closely related to Bell's inequalities in quantum mechanics; see
  \textcite{fritzBellTheoremCorrelation2012a} for some examples.
\item When $\gG$ is an ADMG, $\text{M} \Rightarrow \text{NM}$ is
  generally not true. A well known example is the ``Verma
  constraint''
  \parencite{vermaEquivalenceSynthesisCausal1990,richardsonNestedMarkovProperties2023}.
\item When $\gG$ is an ADMG, $\text{NM} \Rightarrow \text{NE}$ is
  generally not true. This is because $\text{NM}$ only contains
  equality constraints on positive finite-state distributions and latent variable models such as the
  $\text{E}$ model may contain inequality constraints. A well known
  example is the Balke-Pearl bound for the instrumental variable
  graph \parencite{balkeBoundsTreatmentEffects1997}.
\end{enumerate}

\subsubsection{Proof of new claims}
\label{sec:new-claims}

\begin{lemma} \label{lem:unconfounded-e-ef}
  For $\gG \in \UnADMGc(V)$ and any product space $\mathbb{V}$, we have
  $\statmodelE(\gG, \mathbb{V}) = \statmodelEF(\gG,\mathbb{V})$.
\end{lemma}
\begin{proof}
  Let $E=V_{\mathcal E}\subseteq V$ be the set of exogenous
  vertices, and write $\pa(j)$ for the parent set of $V_j$.
  Since $E$ is ancestral, $\marg_E(\gG)$ is the induced
  bidirected graph on $E$, on which M and MI coincide.

  If $\P\in\statmodelE(\gG,\mathbb V)$, then
  \eqref{eq:independent-error} makes the endogenous noises
  mutually independent and independent of the exogenous noises.
  The recursive equations therefore give the EF factorization.
  The exogenous variables are coordinatewise functions of their
  noises, so their marginal satisfies MI, and hence M, on
  $\marg_E(\gG)$. Thus $\P\in\statmodelEF(\gG,\mathbb V)$.

  Conversely, let $\P\in\statmodelEF(\gG,\mathbb V)$, so
  \[
    \P(dv)=\P(dv_{\mathcal E})
    \prod_{V_j\notin E}\P(dv_j\mid v_{\pa(j)}).
  \]
  Draw $E$ from its marginal under $\P$ and set
  $E_j'=e_j(V_j)$ for $j\in\mathcal E$, using Borel encodings.
  For $j\notin\mathcal E$, take mutually independent
  $E_j'\sim\operatorname{Unif}[0,1]$, independent of
  $V_{\mathcal E}$, and choose measurable $f_j$ realizing
  $\P(dv_j\mid v_{\pa(j)})$. Then the recursive system
  \[
    V_j=
    \begin{cases}
      e_j^{-1}(E_j'), & j\in\mathcal E,\\
      f_j(V_{\pa(j)},E_j'), & j\notin\mathcal E
    \end{cases}
  \]
  has joint law $\P$, where each inverse is extended measurably
  to all of $[0,1]$.
  The encoded exogenous noises retain the marginal independences
  of $\marg_E(\P)$. Together with the independence of the
  remaining noises, this gives \eqref{eq:independent-error}.
  Hence $\P\in\statmodelE(\gG,\mathbb V)$.
\end{proof}

\begin{lemma} \label{lem:unconfounded-ef-gm}
  For $\gG \in \UnADMGc(V)$ and any product space $\mathbb{V}$, we have
  \[\fitdisplay{\statmodelEF(\gG,\mathbb{V}) = \statmodelGM(\gG,\mathbb{V}).}\]
\end{lemma}
\begin{proof}
  Let $E=V_{\mathcal E}$ be the set of exogenous vertices.
  Choose a topological order $\prec$ with all exogenous vertices
  first, and write $\pa(j)$ for the parent set of $V_j$.

  If $\P\in\statmodelGM(\gG,\mathbb V) = \statmodelLM(\gG,\mathbb V)$, its exogenous marginal
  is Markov on $\marg_E(\gG)$. For each endogenous vertex, the
  ordered local Markov property reduces its conditional law
  given its predecessors to a kernel depending only on its
  parents. The chain rule therefore gives the EF factorization,
  so $\P\in\statmodelEF(\gG,\mathbb V)$.

  We next prove the reverse direction using the augmentation
  criterion. Let $E=V_{\mathcal E}\subseteq V$ denote the set
  of exogenous vertices, and suppose
  $\P\in\statmodelEF(\gG,\mathbb V)$.
  The exogenous marginal is Markov on $\marg_E(\gG)$, so its
  connected components are mutually independent. It is easy to see
  that $\P$ factorizes according to $\augg(\gG)$. By the
  Hammersley-Clifford theorem, we have
  $\P\in\statmodelGM(\augg(\gG),\mathbb V)$.

  Now consider any ancestral set $J=V_{\mathcal J}$.
  The EF factorization can be written as
  \[
    \begin{aligned}
      \P(dv)
      ={}&\P(dv_{\mathcal E\cap\mathcal J})
      \P(dv_{\mathcal E\setminus\mathcal J}
         \mid v_{\mathcal E\cap\mathcal J})
      \prod_{j\in\mathcal J\setminus\mathcal E}
        \P(dv_j\mid v_{\pa(j)})
      \prod_{j\notin\mathcal J\cup\mathcal E}
        \P(dv_j\mid v_{\pa(j)}).
    \end{aligned}
  \]
  Since $J$ is ancestral, integrating out $V \setminus J$ gives
  \[
    \marg_J(\P)(dv_{\mathcal J})
    =
    \P(dv_{\mathcal E\cap\mathcal J})
    \prod_{j\in\mathcal J\setminus\mathcal E}
      \P(dv_j\mid v_{\pa(j)}).
  \]
  The marginal on $E\cap J$ retains the Markov property on
  its induced bidirected graph. Since an ancestral margin
  of an ADMG is its induced subgraph, this shows that
  \[
    \marg_J(\P)\in
    \statmodelEF(\marg_J(\gG),\mathbb V_{\mathcal J}).
  \]
  Applying the preceding argument to this margin gives
  \[
    \marg_J(\P)\in
    \statmodelGM\bigl(\augg(\marg_J(\gG)),
                     \mathbb V_{\mathcal J}\bigr).
  \]
  Therefore, $\P\in\statmodelA(\gG,\mathbb V)$, and hence
  \[
    \statmodelEF(\gG,\mathbb V)
    \subseteq\statmodelA(\gG,\mathbb V)
    =\statmodelGM(\gG,\mathbb V).
  \]
\end{proof}

\begin{lemma} \label{lem:admg-ne-e}
  For $\gG \in \ADMGc(V)$ and any product space $\mathbb{V}$, we have
  $\statmodelNE(\gG, \mathbb{V}) \subseteq
  \statmodelE(\gG,\mathbb{V})$.
\end{lemma}
\begin{proof}
  Consider any distribution $\P
  \in \statmodelNE(\gG, \mathbb{V})$ on $V$. By definition, there
  exists a distribution $\P'$ on $(V, E)$
  such that $\P' \in \statmodelGM\big(\gG', \mathbb{V} \times
  [0,1]^{|V|}\big)$ for $\gG' = \expandN(\gG)$ and $\marg_V(\P') =
  \P$. Because $\gG'$ is exogenously bidirected, $\P'$ must satisfy the exogenous
  factorization property (\Cref{lem:unconfounded-ef-gm}):
  \[\fitdisplay{
    \P '(dv  \mid  e) = \prod_{j=1}^d \P '(dv _j  \mid  v_{\pa(j)}, e_j),
  }\]
  where $\pa(j) = \pa_{\gG}(j)$ contains indices for the parents of
  $V_j$ in $\gG$ and the marginal distribution of $E$ is Markov
   with respect to the bidirected component of $\gG$. Let $E'_1,
  \dots, E'_{d}$ be independent uniform random variables over
  $[0,1]$ (independent of $E$) that realize the conditional kernels
  $\P'(dv_j \mid v_{\pa(j)}, e_j)$ via measurable functions $f_j$, and
  let $V' = (V'_1,\dots,V'_{d})$ be defined recursively by
  \[\fitdisplay{
    V'_j = f _j( V'_{\pa(j)}, E_j, E'_j ),~j=1,\dots,d,
   }\]
  so that $V'$ has the
  same distribution $\P$ as $V$. Let $h: [0,1] \times [0,1] \to [0,1]$
  be a Borel encoding with a measurable inverse on its image.\footnote{One simple construction is
    to alternate between
    the digits in the binary expansion of the two arguments.}
  It is obvious that
  $V'_j$ is a function of $V'_{\pa(j)}$ and $h(E_j,E_j')$, and the
  distribution of $h(E_1,E_1'), \dots, h(E_d,
  E_d')$ is Markov with respect to the bidirected
  component of $\gG$. Thus $\P \in
  \mathbb{P}_{\textnormal{E}}(\gG, \mathbb{V})$.
 \end{proof}

\begin{lemma}
  \label{lem:admg-ce-ne}
  For $\gG \in \ADMGc(V)$ and any product space $\mathbb{V}$, we have
  $\statmodelCE(\gG, \mathbb{V}) \subseteq
  \statmodelNE(\gG,\mathbb{V})$.
\end{lemma}
\begin{proof}
  Consider $\P \in \statmodelCE(\gG, \mathbb{V})$ and let $\gG' =
  \expandC(\gG)$. By definition,
  there exists a distribution $\P' \in \statmodelGM(\gG', \mathbb{V} \times
  [0,1]^{|\mathcal{C}(\gG)|})$ on $V$ and $E_{\sJ}, \sJ \in
  \mathcal{C}(\gG)$ such that $\P = \marg_V(\P')$. %
  Because $\gG'$ is exogenously bidirected, $\P'$ must satisfy the exogenous
  factorization property (\Cref{lem:unconfounded-ef-gm}):
  \[\fitdisplay{
    \P '(dv  \mid  e) = \prod_{j=1}^d \P '(dv _j  \mid  v_{\pa(j)}, \tilde{e}_j),
  }\]
  where $\pa(j)$ is the parent set of $V_j$ in $\gG$ and $\tilde{E}_j
  = (E_{\sJ}: \sJ \in \mathcal{C}(\gG), j \in \sJ)$ collects latent
  variables in $\gG'$ with a directed edge to $V_j$. Because the
  latent variables $(E_{\sJ}: \sJ \in \mathcal{C}(\gG))$ are mutually
  independent under $\P'$, it is easy to see that the distribution of
  $(\tilde{E}_1,\dots,\tilde{E}_d)$ is Markov with respect to the
  bidirected component of $\gG$. Let $h_j$
  be a Borel encoding that maps $[0,1]^{|\tilde{E}_j|}$ into
  $[0,1]$. Thus, the distribution of $(V_1,\dots,V_d,h_1(\tilde{E}_1),
  \dots, h_d(\tilde{E}_d))$ satisfies the exogenous factorization
  property (and thus the global Markov property by
  \Cref{lem:unconfounded-ef-gm}) with respect to the noise expansion
  graph $\expandN(\gG)$. This shows that $\P \in \statmodelNE(\gG,
  \mathbb{V})$.
\end{proof}

\subsection{\texorpdfstring{Proof of \Cref{thm:statmodelE-agnostic}}{Proof of Theorem 2}}
\label{sec:proof-crefthm:st-agn}

For $\gG \in \ADMGc(V)$, let the collection of all canonical ADMG
expansions of $\gG$ be denoted as
 \[\fitdisplay{
  \expand(\gG) =  \bigcup_{V' \supseteq
    V}  \{\gG' \in \ADMGc(V'): \marg_V(\gG') = \gG\},
}\]
all exogenously bidirected expansions of $\gG$ be denoted as
\[\fitdisplay{
  \expand_{\textnormal{E}}(\gG) =  \bigcup_{V' \supseteq
    V}  \{\gG' \in \UnADMGc (V'): \marg_V(\gG') = \gG\},
}\]
and all DAG expansions of $\gG$ be denoted as
\[\fitdisplay{
  \expand_{\textnormal{D}}(\gG) =  \bigcup_{V' \supseteq
    V}  \{\gG' \in \DAG(V'): \marg_V(\gG') = \gG\},
}\]
where the unions are over finite vertex sets $V' \supseteq V$.

\subsubsection{E/NE is complete (for exogenously bidirected expansions)}
\label{sec:noise-expans-agnost}

We first show that the E model (which is equivalent to NE by
\Cref{fig:relation-statmodel-admg} of the main paper) is complete by proving a stronger
result.

\begin{proposition} \label{prop:admg-model}
  For any $\gG \in \ADMGc(V)$, $|V| = d$, and product space
  $\mathbb{V} = \mathbb{V}_1 \times \dots \times \mathbb{V}_d$, we
  have
  \begin{equation*}
    \begin{split}
      \statmodelE(\gG, \mathbb{V}) &= \marg_V\bigl(\statmodelE\bigl(\expandN(\gG),
                                     \mathbb{V} \times
                                     [0,1]^{|V|}\bigr)\bigr) \\
                                   &= \bigcup_{\gG' \in \expand_{\textnormal{E}}(\gG)}
                                     \marg_V\bigl(\statmodelE\bigl(\gG', \mathbb{V} \times [0,1]^{|V(\gG')| -
                                     |V|}\bigr)\bigr) \\
                                   &= \bigcup_{\gG' \in \expand(\gG)}
                                     \marg_V\bigl(\statmodelE\bigl(\gG', \mathbb{V} \times [0,1]^{|V(\gG')| -
                                     |V|}\bigr)\bigr).
    \end{split}
  \end{equation*}
\end{proposition}
\begin{proof}
  We have
  \begin{align*}
    \statmodelE(\gG, \mathbb{V}) =& \statmodelNE (\gG, \mathbb{V}) \tag*{(By \Cref{thm:relation-statmodel})}
    \\
    =&
                                            \marg_V(\statmodelGM (\expandN(\gG),
                                            \mathbb{V} \times [0,1]^{|V|}))
                                            \tag*{(By definition)}
    \\
    =&
       \marg_V(\statmodelE(\expandN(\gG), \mathbb{V} \times
       [0,1]^{|V|})) \tag*{(By \Cref{thm:relation-statmodel})} \\
    \subseteq& \bigcup_{\gG' \in \expand_{\textnormal{E}}(\gG)}
               \marg_V(\statmodelE(\gG', \mathbb{V} \times
               [0,1]^{|V(\gG')| - |V|})) \tag*{(By definition)} \\
    \subseteq& \bigcup_{\gG' \in \expand(\gG)}
               \marg_V(\statmodelE(\gG', \mathbb{V} \times
               [0,1]^{|V(\gG')| - |V|})). \tag*{(By definition)}
  \end{align*}
  It remains to prove that $\statmodelE(\gG, \mathbb{V}) \supseteq
  \marg_V(\statmodelE(\gG', \mathbb{V} \times [0,1]^{|V(\gG')| -
    |V|}))$ for all $\gG' \in \expand(\gG)$. This follows from
  \Cref{lem:margin-E} below.
\end{proof}

\begin{lemma}[Marginalization preserves E-model] \label{lem:margin-E}
  For all $\gG \in \ADMGc(V)$, any product space $\mathbb{V}$, and $\tilde{V} \subseteq V$ that takes
  values in $\tilde{\mathbb{V}} = \marg_{\tilde{V}}(\mathbb{V})$, we
  have
  \begin{equation*}
    \marg_{\tilde{V}}(\mathbb{P}_{\textnormal{E}}(\gG, \mathbb{V})) \subseteq
    \mathbb{P}_{\textnormal{E}}(\marg_{\tilde{V}}(\gG), \tilde{\mathbb{V}}).
  \end{equation*}
\end{lemma}
\begin{proof}
  Because marginalization is associative, it suffices to prove this
  for $\tilde{V} = V \setminus \{V_j\}$ for all $V_j \in V$.
  Consider $\P \in \mathbb{P}_{\textnormal{E}}(\gG, \mathbb{V})$,
  so $V$ satisfies the equations in \eqref{eq:np-system} and $E$
  satisfies \eqref{eq:independent-error}. We need to show that
  $\marg_{\tilde{V}}(\P) \in \statmodelE(\marg_{\tilde{V}}(\gG),
  \tilde{\mathbb{V}})$.

  Consider the
  following modifications of the equations:
  \begin{equation}
    \label{eq:proj-np-system-1}
    V_k =
    \begin{cases}
      f_k(V_{\pa(k)}, E_k), & \text{if}~k \not \in \ch(j)~\text{and}~k
                              \neq j, \\
      f_k(V_{\pa(k) \setminus \{j\}}, f_j(V_{\pa(j)}, E_j), E_k),&
                                                                   \text{if}~k \in \ch(j),
    \end{cases}
  \end{equation}
  where $V_{\pa(k)}$ is the set of parents of $V_k$ and $V_{\ch(j)}$ is
  the set of children of $V_j$ in $\gG$.
  In words, we eliminate $V_j$ by plugging $V_j = f_j(V_{\pa(j)},
  E_j)$ in all the equations for the children of $V_j$ in $\gG$. We
  claim that this results in a nonparametric system with respect to
  $\tilde{\gG} = \marg_{\tilde{V}}(\gG)$:
  \begin{equation}
    \label{eq:proj-np-system-2}
    V_k = \tilde{f}_k(V_{\pa_{\tilde{\gG}}(k)}, \tilde{E}_k),~k \neq j,
  \end{equation}
  where $\pa_{\tilde{\gG}}(k)$ is the parent index set of $V_k$ in $\tilde{\gG}$,
  \[\fitdisplay{
    \tilde{E}_k =
    \begin{cases}
      E_k, & \text{if}~k \not \in \ch(j)~\text{and}~k
             \neq j, \\
      g(E_k,E_j),& \text{if}~k \in \ch(j),
    \end{cases}
  }\]
  and $g$ is any Borel encoding from $[0,1]^2$ to $[0,1]$ (for
  example, $g$ can be defined by interlacing the decimal expansions of
  its two arguments).
  To see this, marginalizing out $V_j$ in $\gG$ introduces the
  directed edges $V_{\pa(j)} \rdedge V_{\ch(j)}$, which are respected
  in the modified equations. Thus, the right hand side of
  \eqref{eq:proj-np-system-2} collects all the variables on the right
  hand side of \eqref{eq:proj-np-system-1}. It remains to prove that
  $\tilde{E}$ obeys \eqref{eq:independent-error} of the main paper,
  i.e.\ is Markov with respect to the bidirected component of
  $\tilde{\gG}$.

  Consider disjoint $J = V_{\sJ}, K = V_{\sK} \subset \tilde{V}$ such that
   \begin{equation}
    \label{eq:proj-np-system-bd-sep-g-tilde}
    \textnot  J \bdedge  K \ingraph{\tilde{\gG}}.
  \end{equation}
  Because all bidirected edges in $\gG$ between vertices in
  $\tilde{V}$ are contained in $\tilde{\gG}$, it follows that
  \begin{equation}
    \label{eq:proj-np-system-bd-sep-g}
    \textnot  J \bdedge  K \ingraph{\gG},
  \end{equation}
  and hence $E_{\sJ} \independent E_{\sK}$ by
  \eqref{eq:independent-error}. By construction,
  $\tilde{E}_{\sJ} = E_{\sJ}$ if $\sJ \cap \ch(j) =
  \emptyset$, and $\tilde{E}_{\sJ}$ is a measurable function of $E_{\sJ
    \cup \{j\}}$ if $\sJ \cap \ch(j) \neq \emptyset$ (and similarly for
  $\tilde{E}_{\sK}$).

  We prove $\tilde{E}_{\sJ} \independent
  \tilde{E}_{\sK}$ by considering the following cases:
  \begin{enumerate}
  \item $\sJ \cap \ch(j) = \sK \cap \ch(j) = \emptyset$. Then
    $\tilde{E}_{\sJ} \independent \tilde{E}_{\sK}$
     immediately follows from \eqref{eq:proj-np-system-bd-sep-g} and
    \eqref{eq:independent-error}.
  \item $\sJ \cap \ch(j) \neq \emptyset$, $\sK \cap \ch(j) =
    \emptyset$. We claim that
    \[\fitdisplay{
      \textnot  V_{\sJ \cup \{j\}} \bdedge  K \ingraph{\gG},
    }\]
    otherwise by \eqref{eq:proj-np-system-bd-sep-g} there exists a
    confounding arc $J \ldedge V_j \bdedge K \ingraph{\gG}$ that
    marginalizes to $J \bdedge K \ingraph{\tilde{\gG}}$, which
    contradicts \eqref{eq:proj-np-system-bd-sep-g-tilde}. By
    \eqref{eq:independent-error} of the main paper in $\gG$, we obtain $E_{\sJ \cup
      \{j\}} \independent E_{\sK}$, and hence $\tilde{E}_{\sJ}
    \independent \tilde{E}_{\sK}$.
  \item $\sJ \cap \ch(j) = \emptyset$, $\sK \cap \ch(j) \neq
    \emptyset$. This is symmetric to the previous case.
  \item $\sJ \cap \ch(j) \neq \emptyset$, $\sK \cap \ch(j) \neq
    \emptyset$. This is not possible: choosing a child in each of $J$
    and $K$ gives the confounding arc $J \ldedge V_j \rdedge K
    \ingraph{\gG}$, which implies $J \bdedge K
    \ingraph{\tilde\gG}$ and contradicts
    \eqref{eq:proj-np-system-bd-sep-g-tilde}.
  \end{enumerate}
  This completes our proof of \Cref{lem:margin-E}. %
\end{proof}

\subsubsection{Clique expansion is complete (for DAG and exogenously bidirected expansions)}
\label{sec:clique-expans-model}

Next, we prove that the CE model for ADMGs is the completion of the CE
model for exogenously bidirected ADMGs.

\begin{proposition} \label{prop:CE-agnostic}
  For any $\gG \in \ADMGc(V)$ and product space
  $\mathbb{V}$, we have
  \begin{equation*}
    \begin{split}
      \statmodelCE(\gG, \mathbb{V}) &= \marg_V\bigl(\statmodelCE\bigl(\expandC(\gG),
                                      \mathbb{V} \times
                                      [0,1]^{|\mathcal{C}(\gG)|} \bigr)\bigr) \\
                                    &= \bigcup_{\gG' \in \expand_{\textnormal{D}}(\gG)}
                                       \marg_V\bigl(\statmodelCE\bigl(\gG', \mathbb{V} \times [0,1]^{|V(\gG')| -
                                      |V|}\bigr)\bigr) \\
                                    &= \bigcup_{\gG' \in \expand_{\textnormal{E}}(\gG)}
                                      \marg_V\bigl(\statmodelCE\bigl(\gG', \mathbb{V} \times [0,1]^{|V(\gG')| -
                                      |V|}\bigr)\bigr) \\
                                    &= \bigcup_{\gG' \in \expand(\gG)}
                                      \marg_V\bigl(\statmodelCE\bigl(\gG', \mathbb{V} \times [0,1]^{|V(\gG')| -
                                      |V|}\bigr)\bigr).
    \end{split}
  \end{equation*}
\end{proposition}
\begin{proof}
  The proof is similar to that of \Cref{prop:admg-model}.
  Recall that $\expandC(\gG)$ is always a DAG, so
  \begin{align*}
    \statmodelCE(\gG, \mathbb{V}) =&
                                     \marg_V(\statmodelGM (\expandC(\gG),
                                     \mathbb{V} \times [0,1]^{|\mathcal{C}(\gG)|} ))
                                     \tag*{(By definition)}
    \\
    =&
       \marg_V(\statmodelCE(\expandC(\gG), \mathbb{V} \times
       [0,1]^{|\mathcal{C}(\gG)|} )) \tag*{(By \Cref{thm:relation-statmodel})} \\
    \subseteq& \bigcup_{\gG' \in \expand_{\textnormal{D}}(\gG)}
                \marg_V(\statmodelCE(\gG', \mathbb{V} \times
               [0,1]^{|V(\gG')| - |V|})) \tag*{(By definition)} \\
    \subseteq& \bigcup_{\gG' \in \expand_{\textnormal{E}}(\gG)}
               \marg_V(\statmodelCE(\gG', \mathbb{V} \times
               [0,1]^{|V(\gG')| - |V|})) \tag*{(By definition)} \\
    \subseteq& \bigcup_{\gG' \in \expand(\gG)}
               \marg_V(\statmodelCE(\gG', \mathbb{V} \times
               [0,1]^{|V(\gG')| - |V|})). \tag*{(By definition)}
  \end{align*}
  The reverse direction follows from \Cref{lem:margin-CE} below.
\end{proof}

\begin{lemma}[Marginalization preserves CE model] \label{lem:margin-CE}
  For all $\gG \in \ADMGc(V)$, any product space $\mathbb{V}$, and $\tilde{V} \subseteq V$ that takes
  values in $\tilde{\mathbb{V}} = \marg_{\tilde{V}}(\mathbb{V})$, we have
  \begin{equation*}
    \marg_{\tilde{V}}(\statmodelCE(\gG, \mathbb{V})) \subseteq
    \statmodelCE(\marg_{\tilde{V}}(\gG), \tilde{\mathbb{V}}).
  \end{equation*}
\end{lemma}
\begin{proof}
  As in the proof of \Cref{lem:margin-E}, it suffices to prove
  this for $\tilde{V} = V \setminus \{V_j\}$ for all $V_j \in
  V$. Let $\tilde{\gG} = \marg_{\tilde{V}}(\gG)$.

  Let $\P \in \statmodelCE(\gG, \mathbb{V})$, so by definition,
  there exists $\P' \in \statmodelGM(\expandC(\gG), \mathbb{V} \times
  [0,1]^{|\mathcal{C}(\gG)|})$ such that $\P = \marg_V(\P')$.
  Because
  $\expandC(\gG)$ is a DAG, $\P'$ factorizes over $\expandC(\gG)$ (by
  \Cref{thm:relation-statmodel}). Realizing each conditional kernel at
  $V_k$ with an independent uniform randomization variable $U_k \sim
  \text{Unif}[0,1]$ and absorbing $U_k$ into the singleton-clique root
  $E_{\{k\}}$ via a Borel encoding $[0,1]^2 \to [0,1]$, we can write
  $V$ as a nonparametric system of equations
   \[\fitdisplay{
    V_k = f_k(V_{\pa_{\gG}(k)}, C_k),~k=1,\dots,d
  }\]
  for some measurable functions $f_1,\dots,f_d$, where $C_k = (E_{\sJ}: k \in \sJ \in
  \mathcal{C}(\gG))$ and $E_{\sJ}, \sJ \in \mathcal{C}(\gG)$ are
  mutually independent random variables over $[0,1]$. We would like to show that $\marg_{\tilde{V}}(\P) \in
  \statmodelCE(\tilde{\gG}, \tilde{\mathbb{V}})$, which requires us to
  rewrite the equations as
  \begin{equation}
    \label{eq:proj-np-system-4}
    V_k = \tilde{f}_k(V_{\pa_{\tilde{\gG}}(k)}, \tilde{C}_k),~k \neq j,
  \end{equation}
  where $\pa_{\tilde{\gG}}(k)$ is the parent index set of $V_k$ in $\tilde{\gG}$,
  $\tilde{C}_k = (\tilde{E}_{\tilde{\mathcal{J}}}: k \in
  \tilde{\mathcal{J}} \in \mathcal{C}(\tilde{\gG}))$, and
  $\tilde{E}_{\tilde{\mathcal{J}}}$, $\tilde{\mathcal{J}} \in
  \mathcal{C}(\tilde{\gG})$ are independent.

  For each $\mathcal J\in\mathcal C(\gG)$, define
  \[
    h(\mathcal J)=
    \begin{cases}
      \mathcal J, & j\notin\mathcal J,\\
      (\mathcal J\setminus\{j\})\cup\ch_{\gG}(j),
        & j\in\mathcal J.
    \end{cases}
  \]
  A clique not containing $j$ remains a clique after
  marginalization. A clique containing $j$ remains a clique
  after removing $j$ and adding its children. Thus
  $h(\mathcal J)\in\mathcal C(\tilde{\gG})$ whenever
  $h(\mathcal J)$ is nonempty.

  For each $\tilde{\mathcal J}\in\mathcal C(\tilde{\gG})$,
  combine all original noises whose cliques map to
  $\tilde{\mathcal J}$ into a single variable:
  \[
    \tilde E_{\tilde{\mathcal J}}
    =
    g_{\tilde{\mathcal J}}
    \bigl(E_{\mathcal J}:
      \mathcal J\in\mathcal C(\gG),\
      h(\mathcal J)=\tilde{\mathcal J}\bigr).
  \]
  Here $g_{\tilde{\mathcal J}}$ is a Borel encoding of the
  displayed vector. If there are no such variables, set
  $\tilde E_{\tilde{\mathcal J}}=0$ instead.
  Noises with $h(\mathcal J)=\emptyset$ are discarded.
  Each original noise can be recovered
  from its encoded variable. Moreover, the encoded variables
  are mutually independent because they use disjoint
  collections of the original independent noises.
  See \Cref{ex:dag-agnostic} for an illustration.

  Now eliminate $V_j$ by substituting its equation into
  the equations for its children. We verify
  \eqref{eq:proj-np-system-4} in two cases:
  \begin{enumerate}
  \item If $k\notin\ch_{\gG}(j)$ and $k\ne j$, then
  $\pa_{\tilde{\gG}}(k)=\pa_{\gG}(k)$.
  Every noise $E_{\mathcal J}$ in $C_k$ satisfies
  $k\in\mathcal J$, hence $k\in h(\mathcal J)$.
  It can therefore be recovered from
  $\tilde E_{h(\mathcal J)}$ in $\tilde C_k$.

  \item If $k\in\ch_{\gG}(j)$, then
  \[
    \pa_{\tilde{\gG}}(k)
    =
    (\pa_{\gG}(k)\setminus\{j\})\cup\pa_{\gG}(j).
  \]
  After substitution, the equation for $V_k$ uses only
  these parents and the noises in $C_j$ and $C_k$.
  The noises in $C_k$ are recoverable as above.
  For every noise $E_{\mathcal J}$ in $C_j$, we have
  $j\in\mathcal J$, so
  $\ch_{\gG}(j)\subseteq h(\mathcal J)$.
  Thus this noise is also recoverable from $\tilde C_k$.
  In particular, every child recovers the same original
  noises in $C_j$ and hence the same value of $V_j$.
  \end{enumerate}

  This gives \eqref{eq:proj-np-system-4} with independent
  clique noises, proving
  $\marg_{\tilde V}(\P)\in
  \statmodelCE(\tilde{\gG},\tilde{\mathbb V})$.
\end{proof}

  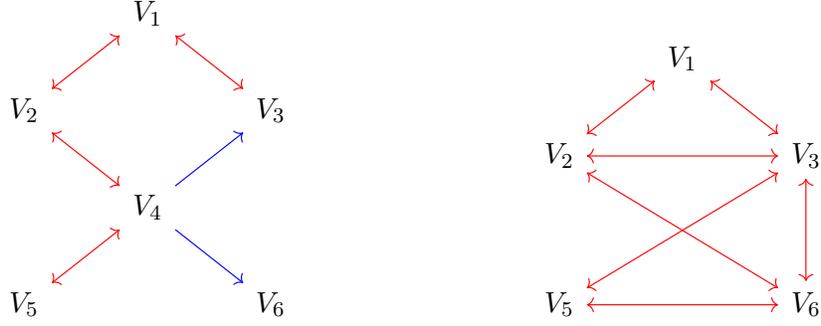
\begin{figure}[t]
    \centering
    \begin{subfigure}[b]{0.4\textwidth}
      \[
        \begin{tikzcd}
          & V_1 \arrow[ld, leftrightarrow, red] \arrow[rd, leftrightarrow, red] & \\
          V_2 \arrow[rd, leftrightarrow, red] & & V_3 \\
          & V_4 \arrow[ld, leftrightarrow, red] \arrow[ur, blue] \arrow[dr,
          blue] & \\
          V_5 & & V_6
        \end{tikzcd}
      \]
      \caption{Cliques: 1, 2, 3, 4, 5, 6, 12, 13, 24, 45.\\ \phantom{}}
      \label{fig:new-clique-a}
    \end{subfigure} \quad
    \begin{subfigure}[b]{0.4\textwidth}
      \[
        \begin{tikzcd}
          & V_1 \arrow[ld, leftrightarrow, red] \arrow[rd, leftrightarrow, red] & \\
          V_2 \arrow[rr, leftrightarrow, red] \arrow[rrdd, leftrightarrow,
          red] & & V_3 \arrow[dd, leftrightarrow, red] \\
          & & \\
          V_5 \arrow[rr, leftrightarrow, red] \arrow[rruu, leftrightarrow, red] & & V_6
        \end{tikzcd}
      \]
      \caption{Cliques: 1, 2, 3, 5, 6, 12, 13, 23, 26, 35, 36, 56,
        123, 236, 356.}
      \label{fig:new-clique-b}
    \end{subfigure}
    \caption{Marginalization can create many new cliques.}
    \label{fig:new-clique}
  \end{figure}

\begin{example} \label{ex:dag-agnostic}
  As an example to illustrate the construction of the proof above, let
  $\gG$ be the graph in \Cref{fig:new-clique-a}. After substituting
  $V_4=f_4(E_4,E_{24},E_{45})$, the remaining equations of its clique
  expansion can be rewritten using the clique roots of
  $\marg_{V\setminus\{V_4\}}(\gG)$, shown in \Cref{fig:new-clique-b},
  as follows:
\begin{alignat*}{2}
 V_1 &= f_1(E_1, E_{12}, E_{13}) &&= \tilde{f}_1(\tilde{E}_1, \tilde{E}_{12},
                                    \tilde{E}_{13}, \tilde{E}_{123}), \\
 V_2 &= f_2(E_2, E_{12}, E_{24}) &&= \tilde{f}_2(\tilde{E}_2, \tilde{E}_{12},
                               \tilde{E}_{23}, \tilde{E}_{26} ,
                                    \tilde{E}_{123}, \tilde{E}_{236}),
  \\
 V_3 &= f_3(f_4(E_4, E_{24}, E_{45}), E_3, E_{13}) &&= \tilde{f}_3(\tilde{E}_3, \tilde{E}_{13},
                                    \tilde{E}_{23}, \tilde{E}_{35},
                            \tilde{E}_{36}, \tilde{E}_{123},
                                                      \tilde{E}_{236},
                                                      \tilde{E}_{356}), \\
 V_5 &= f_5(E_5, E_{45}) &&= \tilde{f}_5(\tilde{E}_5, \tilde{E}_{35},
                               \tilde{E}_{56},
                                    \tilde{E}_{356}),
  \\
 V_6 &= f_6(f_4(E_4, E_{24}, E_{45}), E_6) &&= \tilde{f}_6(\tilde{E}_6, \tilde{E}_{26},
                                    \tilde{E}_{36}, \tilde{E}_{56},
                                                      \tilde{E}_{236},
                                                      \tilde{E}_{356}),
\end{alignat*}
where $\tilde{E}_{\cdot} = E_{\cdot}$ for $\cdot \in
\{1,2,3,5,6, 12, 13\}$, $\tilde{E}_{36} = E_4$, $\tilde{E}_{236} =
E_{24}$, $\tilde{E}_{356} = E_{45}$, and $\tilde{E}_{\cdot} = 0$ for
$\cdot \in \{23, 26, 35, 56, 123\}$.
\end{example}

\subsubsection{MI is complete (for DAG and exogenously bidirected expansions)}
\label{sec:um-agnostic}

Because the MI model is different from common interpretations of
ADMGs, let us use an example to warm up.
\begin{example}
Consider the instrumental variable graph
\[
  \begin{tikzcd}
    Z \arrow[r, blue] & X \arrow[r, blue] \arrow[r, red,
    leftrightarrow, bend left] & Y
  \end{tikzcd}
\]
and its clique expansion
 \[
  \begin{tikzcd}
    & & E_{XY} \arrow[ld, blue] \arrow[rd, blue] & \\
    Z \arrow[r, blue] & X \arrow[rr, blue] & & Y
  \end{tikzcd}.
\]
The MI model for the instrumental variable graph contains all
probability distributions of $(Z, X, Y)$ because they are all
connected by arcs. However, it is well known that a latent
variable
interpretation of this graph imposes inequality constraints
\parencite{balkeBoundsTreatmentEffects1997}. The latent variable
interpretation implicitly
assumes the usual interpretation of DAGs (e.g.\ factorization, Markov, or any of their non-NM equivalences in
\Cref{fig:relation-statmodel-dag}), which contains all distributions
of $(Z, X, Y, E_{XY})$ such that $Z \independent E_{XY}$ and $Y
\independent Z \mid X, E_{XY}$. In contrast, the MI model
contains all distributions of $(Z, X, Y, E_{XY})$ such that $E_{XY}
\independent Z$, which impose no constraint on the marginal
distribution of $(Z, X, Y)$.
\end{example}

We now prove that the MI model is complete with respect to
DAG and exogenously bidirected graph expansions. First, we show
\[\fitdisplay{
  \statmodelMI (\gG, \mathbb{V} ) \subseteq \bigcup_{\gG' \in
    \expand_{\textnormal{D}}(\gG)} \marg_V\bigl(\statmodelMI\bigl( \gG', \mathbb{V} \times [0,1]^{|V(\gG')|-|V|}\bigr)\bigr)
 }\]
with an almost trivial construction. Consider any $\P \in
\statmodelMI(\gG, \mathbb{V})$. Consider the clique
expansion $\expandC(\gG)$ and the product probability distribution
 \[\fitdisplay{
  \P '  = \P \otimes \Q ,
}\]
where $\otimes$ denotes the product of probability measures and $\Q$ is
 the uniform distribution over $[0,1]^{|\mathcal{C}(\gG)|}$. It is
obvious that $\P'$ marginalizes to $\P$ and $\P'$ satisfies the
marginal independence property with respect to the clique expansion
graph.

The reverse direction follows from the fact that marginalization
preserves m-connection. That is, for disjoint $J, K \subseteq V
\subseteq V'$ and graphs $\gG \in \ADMGc(V), \gG' \in \ADMGc(V')$, if
$\marg_V(\gG') = \gG$, then $J \mconnarc K \ingraph{\gG}$ if and only
if $J \mconnarc K \ingraph{\gG'}$ \parencite[see, for
example,][Theorem 2]{guoConfounderSelectionIterative2023}.

\subsubsection{Other ADMG models are not complete}
\label{sec:other-admg-models}

Because E is equivalent to M (as well as EF, LM, and A) on
exogenously bidirected graphs, the completion of M with respect
to such expansions is E by \Cref{prop:admg-model}. Its completion
with respect to DAG expansions is CE by \Cref{prop:CE-agnostic}.
Under the convention in \Cref{def:agnostic}, the same completions
apply to NM on finite or countable target spaces.
By \Cref{thm:relation-statmodel}, these completions differ from
M and NM in general. Thus NM, M, LM, and A are not complete with
respect to exogenously bidirected or DAG expansions.
Similarly, E and NE coincide with M on DAGs, so their DAG
completion is CE, which differs from both in general.
Hence E and NE are not complete with respect to DAG expansions.

It remains to show that PE is not complete. Consider the
``bidirected 3-cycle'' with edges $A \bdedge B$, $B \bdedge C$, and $C
\bdedge A$. If the PE model is complete with respect to DAG or
exogenously bidirected expansions, it should contain the
$(A,B,C)$-marginal of the DAG with edges $U \rdedge A$, $U \rdedge B$,
$U \rdedge C$, which places no restrictions on the distribution of
$(A,B,C)$. However, the PE model has some inequality constraints; see
\textcite[Example 2.11]{fritzBellTheoremCorrelation2012a}. So the PE
model is ``too small''.

\subsection{\texorpdfstring{Proof of \Cref{thm:equivalent-axioms}}{Proof of Theorem 4}}
\label{sec:proof-crefl-po}

Throughout this proof we assume the state space $\mathbb V$ is finite or countable.
The forward implication follows from the next three lemmas.

\begin{lemma}[Consistency]\label{lem:po-consistency}
For $\gG\in\ADMGc(V)$,
$\causalmodel(\gG,\mathbb V)\subseteq\causalmodel(\mathbb V)$.
\end{lemma}
\begin{proof}
  Let $\P\in\causalmodel(\gG,\mathbb V)$.
   Suppose $V_{\sI'}(v_{\sI}) = v_{\sI'}$. Consider any $V_j \in V$. It
  follows from \eqref{eq:basic-consistency} that
  \[\fitdisplay{\begin{aligned}
    V_j(v_{\sI}, v_{\sI'}) &= V_j(v_{\pa(j) \cap \sI}, v_{\pa(j) \cap \sI'},
                             V_{\pa(j) \setminus (\sI \cup
                             \sI')}(v_{\sI},v_{\sI'}))
  \end{aligned}}\]
  and
  \[\fitdisplay{\begin{aligned}
    V_j(v_{\sI}) &= V_j(v_{\pa(j) \cap \sI}, V_{\pa(j) \cap \sI'}(v_{\sI}),
                   V_{\pa(j) \setminus (\sI \cup \sI')}(v_{\sI})) \\
    &= V_j(v_{\pa(j) \cap \sI}, v_{\pa(j) \cap \sI'},
    V_{\pa(j) \setminus (\sI \cup \sI')}(v_{\sI})).
  \end{aligned}}\]
  Thus, it suffices to show that, if $\pa(j) \setminus (\sI \cup
  \sI')$ is not empty,
  \[\fitdisplay{
    V_{\pa(j) \setminus (\sI \cup \sI')}(v_{\sI},v_{\sI'}) = V_{\pa(j)
      \setminus (\sI \cup \sI')}(v_{\sI}).
  }\]
  The proof can then be completed by an induction argument.
\end{proof}

\begin{lemma}[Causal irrelevance]\label{lem:po-irrelevance}
Let $\P\in\causalmodel(\gG,\mathbb V)$. For disjoint
$\sK,\sL\subseteq[d]$ and any $\sJ\subseteq[d]$, if
$\textnot V_{\sL}\rdpath V_{\sJ}\mid V_{\sK}\ingraph{\gG}$, then
\[\fitdisplay{
 \P\{V_{\sJ}(v_{\sK},v_{\sL})=V_{\sJ}(v_{\sK})\}=1, \quad \text{for
 all}~v_{\sK}\in\mathbb V_{\sK}~\text{and}~v_{\sL}\in\mathbb V_{\sL}.
}\]
\end{lemma}
\begin{proof}
It suffices to prove the claim for $\sJ = \{j\}$. In this case, it
follows from the condition $\textnot V_{\sL} \rdpath V_j \mid
V_{\sK} \ingraph{\gG}$ that $\pa(j) \cap \sL = \emptyset$ and
$\textnot V_{\sL} \rdpath V_{\pa(j) \setminus \sK} \mid V_{\sK}
\ingraph{\gG}$. By applying the recursive substitution in
\eqref{eq:basic-consistency}, we have
\[\fitdisplay{\begin{aligned}
  \P(V_j(v_{\sK}, v_{\sL}) = V_j(v_{\sK})) &= \P\bigl(V_j(v_{\pa(j) \cap
    \sK}, V_{\pa(j) \setminus \sK}(v_{\sK}, v_{\sL})) = V_j(v_{\pa(j) \cap
                                             \sK}, V_{\pa(j) \setminus \sK}(v_{\sK}))\bigr) \\
  & \geq \P(V_{\pa(j) \setminus \sK}(v_{\sK}, v_{\sL}) = V_{\pa(j) \setminus \sK}(v_{\sK})).
\end{aligned}}\]
The proof can then be completed by an induction argument.
\end{proof}

\begin{lemma}[Single-world Markov property]\label{lem:po-markov}
Let $\P\in\causalmodel(\gG,\mathbb V)$. For every
$\sI\subseteq[d]$ and $v_{\sI}\in\mathbb V_{\sI}$,
\[\fitdisplay{
 \marg_{V(v_{\sI})}(\P)\in\statmodelGM(\gG(v_{\sI}),\mathbb V).
}\]
\end{lemma}
\begin{proof}
Fix $v_{\sI}$. Using the augmentation criterion, it suffices to show that,
for every ancestral set $A=V_{\mathcal A}$ in $\gG(v_{\sI})$, the
distribution of $V_{\mathcal A}(v_{\sI})$ is Markov with respect to
$\augg \circ \marg_A(\gG(v_{\sI}))$, the subgraph of $\gG(v_{\sI})$ on $A$
after augmentation.

For every full assignment $v'\in\mathbb V$, we have
\[\fitdisplay{
 \{V_{\mathcal A}(v_{\sI})=v'_{\mathcal A}\}
 =\{V_{\mathcal A}(v_{\sI},v'_{\sI^c})=v'_{\mathcal A}\}
 \quad\text{almost surely}.
}\]
The equality follows from consistency and causal irrelevance
(\Cref{lem:po-consistency,lem:po-irrelevance}), using ancestrality
of $A$ and recursive substitution in topological order.
Let $D_l=V_{\mathcal D_l}$, $l=1,\ldots,m$, be all the districts in
$\marg_A(\gG(v_{\sI}))$. There are no bidirected edges between them,
so the basic-outcome independence axiom
\eqref{eq:independent-basic-po} gives
\[\fitdisplay{
 \P(V_{\mathcal A}(v_{\sI})=v'_{\mathcal A})
 =\prod_{l=1}^m
 \P\bigl(V_{\mathcal D_l}(v_{\sI},v'_{\sI^c})=v'_{\mathcal D_l}\bigr).
}\]
Since the full intervention varies with $v'$, this is a factorization
of assignment probabilities, not independence of the district
variables under $v_{\sI}$.
By recursive substitution, each district factor depends only on the
proposed outcomes in $D_l\cup\pa_{\marg_A(\gG(v_{\sI}))}(D_l)$, which
forms a clique in $\augg \circ \marg_A(\gG(v_{\sI}))$. The three cases
are visible in the following collider-connected
walks, with $V_j,V_k \in D_l$:
\[\fitdisplay{
\begin{aligned}
 \text{district--district:}\quad &V_j\samedist V_k,\\
 \text{parent--district:}\quad & * \rdedge V_j\samedist V_k,\\
 \text{parent--parent:}\quad & * \rdedge V_j\samedist V_k\ldedge *.
\end{aligned}
}\]
Thus $\marg_{V_{\mathcal{A}}(v_{\sI})}(\P)$ factorizes with respect to
$\augg \circ \marg_A(\gG(v_{\sI}))$. By the Hammersley-Clifford
theorem, it is also Markov with respect to that graph, completing the
proof.
\end{proof}

For the reverse direction in \Cref{thm:equivalent-axioms} suppose $\P$
satisfies the potential-outcome calculus in
\Cref{def:po-calculus}. We next verify the two defining properties of the
causal Markov model. The marginal independence property of basic
potential outcomes in \eqref{eq:independent-basic-po} is a special
case of the single-world Markov property with the full intervention ($\sI
= [d]$). For recursive substitution, we have
\[
\begin{aligned}
  V_j(v_{\sI})
  &=V_j\bigl(v_{\sI},
      V_{\pa(j)\setminus\sI}(v_{\sI})\bigr)\\
  &=V_j\bigl(v_{\pa(j)\cap\sI},
      V_{\pa(j)\setminus\sI}(v_{\sI})\bigr).
\end{aligned}
\]
The first equality follows from consistency
\eqref{eq:consistency}; the second follows from causal
irrelevance, since interventions
outside the parent set are irrelevant once all parent inputs
are fixed. This is \eqref{eq:basic-consistency}.

\subsection{\texorpdfstring{Proof of \Cref{prop:fixing}}{Proof of Proposition 1}}
\label{sec:proof-crefpr}

Let us first prove the following graphical result.

\begin{lemma} \label{lem:fixing-graph}
  A vertex $V_j \in V$ is fixable in $\gG \in \ADMGc(V)$ if and only
  if
  \begin{equation}
    \label{eq:fixing-graph}
    \textnot V_j(v_j) \mconn V_{\de_{\gG}(j)}(v_j) \mid
    V_{\nd_{\gG}(j)}(v_j) \ingraph{\gG(v_j)},
  \end{equation}
  where $\nd_{\gG}(j) = [d] \setminus \{j\} \setminus \de_{\gG}(j)$
  collects the indices of the non-descendants of $V_j$ in $\gG$.
\end{lemma}
\begin{proof}
  Because $V_j(v_j)$, $V_{\de(j)}(v_j)$, and
  $V_{\nd(j)}(v_j)$ give a partition of the vertex set of
  $\gG(v_j)$, the m-separation in \eqref{eq:fixing-graph} is
  equivalent to
  \[\fitdisplay{
    \textnot V_j(v_j) \colliderconn V_{\de(j)}(v_j) \ingraph{\gG(v_j)},
  }\]
  which is further equivalent to
  \[\fitdisplay{
    \textnot V_j(v_j) \samedist V_{\de(j)}(v_j) \ingraph{\gG(v_j)}
  }\]
  because $V_j(v_j)$ has no children in $\gG(v_j)$ and no descendant
  of $V_j$ can point to $V_j$ or a non-descendant. By the
  definition of $\gG(v_j)$, the last condition is equivalent to
  \[\fitdisplay{
    \textnot V_j \samedist V_{\de(j)} \ingraph{\gG},
  }\]
  which is exactly what fixability of $V_j$ means.
\end{proof}

We now prove \Cref{prop:fixing}. By causal irrelevance,
$V_{\nd(j)}(v_j)=V_{\nd(j)}$ and $V_j(v_j)=V_j$ almost surely.
Fixability, \Cref{lem:fixing-graph}, and the single-world Markov
property give
\[
  V_j\independent V_{\de(j)}(v_j)\mid V_{\nd(j)}.
\]
Thus, whenever $\P(v_j,v_{\nd(j)})>0$,
\[
\begin{aligned}
  &\P\bigl(V_j(v_j)=\tilde v_j,\,
           V_{-j}(v_j)=v_{-j}\bigr)\\
  &\quad=
  \P(v_{\nd(j)})\P(\tilde v_j\mid v_{\nd(j)})
  \P\bigl(V_{\de(j)}(v_j)=v_{\de(j)}
          \mid V_{\nd(j)}=v_{\nd(j)}\bigr)\\
  &\quad=
  \P(v_{\nd(j)})\P(\tilde v_j\mid v_{\nd(j)})
  \P(v_{\de(j)}\mid v_j,v_{\nd(j)}),
\end{aligned}
\]
where the last equality follows from the displayed independence
and consistency. Factorizing $\P(v_j,v_{-j})$ similarly gives
\begin{equation}
  \label{eq:fixing-proof}
  \P\bigl(V_j(v_j)=\tilde v_j,\,
          V_{-j}(v_j)=v_{-j}\bigr)
  =
  \P(v_j,v_{-j})
  \frac{\P(\tilde v_j\mid v_{\nd(j)})}
       {\P(v_j\mid v_{\nd(j)})}.
\end{equation}

Since $V_j$ is fixable,
$\mathcal B=\mbg_{\gG}(j)\subseteq\nd(j)$ is its Markov
boundary in the ancestral set $V_{\nd(j)}\cup\{V_j\}$.
The observational Markov property therefore gives
\[
  V_j\independent V_{\nd(j)\setminus\mathcal B}
  \mid V_{\mathcal B}.
\]
Whenever $\P(v_{\nd(j)})>0$, we may consequently replace
conditioning on $v_{\nd(j)}$ by conditioning on
$v_{\mathcal B}$. If also $\P(v_j\mid v_{\mathcal B})>0$,
then $\P(v_j,v_{\nd(j)})>0$, so \eqref{eq:fixing-proof}
gives \eqref{eq:fixing-intervention}.

Now suppose $\P(v_{\nd(j)})=0$. The left-hand side of
\eqref{eq:fixing-intervention} is zero because
$V_{\nd(j)}(v_j)=V_{\nd(j)}$ almost surely. Also,
$\P(v_j,v_{-j})=0$. There are two cases covered by the
proposition:
\begin{itemize}
  \item If $\P(v_{\mathcal B})>0$ and
    $\P(v_j\mid v_{\mathcal B})>0$, the ratio on the
    right-hand side is well defined, so that side is also zero.
  \item If $\P(v_{\mathcal B})=0$, the right-hand side is
    zero by the convention stated in the proposition.
\end{itemize}
No assertion is made when $\P(v_{\mathcal B})>0$ but
$\P(v_j\mid v_{\mathcal B})=0$; in that case, the
intervention probability need not vanish.

\subsection{\texorpdfstring{Proof of \Cref{thm:admg-ne-nm}}{Proof of Theorem 3}}
\label{sec:proof-crefthm:-ne}

Let $\P_V\in\statmodelE(\gG,\mathbb V)$ and choose a structural
representation with associated causal law $\P$ as in
\Cref{sec:proof-sketch-crefthm-ne}. To justify
\eqref{eq:invariance-fixing}
also at zero-probability assignments, we specify the conditional
versions used at each fixing step.

Consider first fixing $V_j=v_j$. When
$\P(v_j,v_{\nd(j)})>0$, the argument of \Cref{prop:fixing}
applies. When
$\P(v_j,v_{\nd(j)})=0<\P(v_{\nd(j)})$, the observational
conditional is undetermined; choose
\[
  \P(v_{\de(j)}\mid v_j,v_{\nd(j)})
  =
  \P\bigl(V_{\de(j)}(v_j)=v_{\de(j)}
          \mid V_{\nd(j)}=v_{\nd(j)}\bigr).
\]
When $\P(v_{\nd(j)})=0$, choose any normalized conditional
distribution, since its multiplying marginal probability
in \eqref{eq:conditional-product-fixing} is zero.
Since $V_{\nd(j)}(v_j)=V_{\nd(j)}$, these choices give
\[
\begin{aligned}
  \bigl(\fix_{V_j=v_j}(\P_V)\bigr)(v_{-j})
  &=\P(v_{\nd(j)})
    \P(v_{\de(j)}\mid v_j,v_{\nd(j)})\\
  &=\P\bigl(V_{-j}(v_j)=v_{-j}\bigr)
\end{aligned}
\]
for every assignment. Apply the same construction to the
residual structural system after each subsequent intervention.
Using one causal representation throughout makes the resulting
law depend only on the fixed set and its values, establishing
\eqref{eq:invariance-fixing} for every valid fixing sequence.

By \textcite[Theorem 16]{richardsonNestedMarkovProperties2023},
it suffices to verify the ordered local Markov property (for extended
conditional independence given before \Cref{def:nested-markov}).
Consider a valid fixing sequence $J=V_{\sJ}$ and write
$\tilde{\gG}=\fix_J(\gG)$. Choose a topological order of $\tilde{\gG}$
in which the fixed vertices $J$ precede $V\setminus J$ ($J$ has no
parent in $\tilde{\gG}$). For any $L=V_{\sL}\subseteq V\setminus J$ such
that $\bar{L} = L\cup J$ is ancestral, let $V_k$ be the maximal vertex
in $L$, and let $\tilde{\gG}_{\bar{L}}$ be the induced graph on
$\bar{L}$. We must show
\begin{equation}
  \label{eq:local-nested-markov}
  V_k\independent
  V_{(\sL\cup\sJ)\setminus
      (\mbg_{\tilde{\gG}_{\bar{L}}}(k)\cup\{k\})}
  \mid V_{\mbg_{\tilde{\gG}_{\bar{L}}}(k)}
  \underdist{\fix_J(\P_V)}.
\end{equation}

On positive-probability conditioning events, we have
\[\fitdisplay{\begin{aligned}
  &\fix_{V_{\sJ}=v_{\sJ}}(\P_V)
    (v_k\mid v_{\sL\setminus\{k\}})\\
  &=\P\bigl(V_k(v_{\sJ})=v_k\mid
      V_{\sL\setminus\{k\}}(v_{\sJ})
      =v_{\sL\setminus\{k\}}\bigr)\\
  &=\P\bigl(
      V_k(v_{(\sJ\cup\sL)\setminus\{k\}})=v_k
      \mid
      V_{\sL\setminus\{k\}}
        (v_{(\sJ\cup\sL)\setminus\{k\}})
      =v_{\sL\setminus\{k\}}\bigr)\\
  &=\P\bigl(
      V_k(v_{(\sJ\cup\sL)\setminus\{k\}})=v_k
      \mid
      V_{\dis_{\tilde{\gG}_{\bar{L}}}(k)\setminus\{k\}}
        (v_{(\sJ\cup\sL)\setminus\{k\}})
      =v_{\dis_{\tilde{\gG}_{\bar{L}}}(k)\setminus\{k\}}\bigr)\\
  &=\P\bigl(
      V_k(v_{\pa_{\gG}(k)})=v_k
      \mid
      (V_c(v_{\pa_{\gG}(c)}):
         c\in\dis_{\tilde{\gG}_{\bar{L}}}(k)\setminus\{k\})
      =v_{\dis_{\tilde{\gG}_{\bar{L}}}(k)\setminus\{k\}}\bigr).
\end{aligned}}\]
The first equality follows from \eqref{eq:invariance-fixing},
and the second from consistency and recursive substitution.
The third follows from the corresponding m-separation and
the single-world Markov property. Indeed, ancestrality and
maximality ensure that all parents of $L$ belong to
$V_{(\sJ\cup\sL)\setminus\{k\}}$, so no directed edge enters
$L$ under this intervention. Internal walks cannot cross
districts, while any excursion outside $L$ is blocked by a
collider with no descendants in $L$.
The last equality follows from recursive consistency.

The final expression depends only on $v_k$ and
$v_{\mbg_{\tilde{\gG}_{\bar L}}(k)}$, proving
\eqref{eq:local-nested-markov} on positive-probability
conditioning events. Ancestral sets omitting fixed vertices
give the same conditions: an omitted fixed vertex cannot
parent a retained random vertex, so adjoining it does not
change the relevant Markov background.
By \textcite[Theorem 16]{richardsonNestedMarkovProperties2023},
the global Markov property therefore holds wherever the
conditional kernels are defined.

To cover zero-probability conditioning events, consider any
m-separation of $V_{\sK}$ and $V_{\sL}$ given
$V_{\sM}$ in \eqref{eq:cadmg-global-markov}. Since the fixed vertices $J$ form a bidirected clique in
$\widetilde{\fix}_J(\gG)$, the assumed m-separation implies
that at least one of $V_{\sK}$ and $V_{\sL}$ is disjoint
from $J$. Without loss of generality, suppose $V_{\sK}\cap J=\emptyset$.
For each $v_{\sM}$, the conditional distribution of $V_{\sK}$
is the same for all choices of the other conditioning and
fixed values whenever the conditioning event has positive
probability. Write this distribution as
$g(v_{\sK}\mid v_{\sM})$, and use it also on null conditioning
events with the same $v_{\sM}$. If no conditioning event
with these $v_{\sM}$ values has positive probability, choose
$g(\cdot\mid v_{\sM})$ to be any fixed probability mass
function on $\mathbb V_{\sK}$.
With these choices,
\[\begin{aligned}
  &\fix_{V_{\sJ}=v_{\sJ}}(\P_V)
    (v_{\sK},v_{(\sL\cup\sM)\setminus\sJ})
  =
    g(v_{\sK}\mid v_{\sM})\,
    \fix_{V_{\sJ}=v_{\sJ}}(\P_V)
    (v_{(\sL\cup\sM)\setminus\sJ})
\end{aligned}\]
at every assignment. When the marginal probability on the
right is positive, this follows from the conditional
independence already proved. When it is zero, the joint
probability on the left is also zero, so the identity still
holds. Thus the conditional distribution of $V_{\sK}$ can
be chosen to depend only on $v_{\sM}$, even on null
conditioning events. This establishes
\eqref{eq:cadmg-global-markov} for every valid fixing sequence,
and hence $\P_V\in\statmodelNM(\gG,\mathbb V)$.